\documentclass[times,sort&compress,3p]{elsarticle}
\journal{Journal of Multivariate Analysis}
\usepackage[labelfont=bf]{caption}

\usepackage{amsmath,amsfonts,amssymb,amsthm,booktabs,color,epsfig,graphicx,hyperref,url}

\usepackage{mathrsfs}
\RequirePackage[OT1]{fontenc}

\allowdisplaybreaks 

\numberwithin{equation}{section}

\theoremstyle{plain}
\newtheorem{theorem}{Theorem}
\newtheorem{lemma}{Lemma}

\newtheorem{proposition}{Proposition}

\theoremstyle{definition} 

\newtheorem{remark}{Remark}
\newtheorem{example}{Example}

\newcommand{\ignore}[1]{} 

\font\msbmx=msbm10                   
\font\msbmvii=msbm7                  
\font\msbmv=msbm5
\def\varnothing{\mathchoice{\mbox{\msbmx\char'077}}%
{\mbox{\msbmx\char'077}}{\mbox{\msbmvii\char'077}}{\mbox{\msbmv\char'077}}}%

\def\vect{\mathop{\rm Vect}\limits}
\def\tr{\mathop{\rm tr}\limits}

\newcommand{\vae}{\varepsilon}

\newcommand{\wt}{\widetilde}

\def\bbr{{\mathbb R}}

\newcommand{\mb}[1]{\mathbf{#1}} 

\newcommand{\Xb}{{\mb{X}}}

\newcommand{\mbs}[1]{\boldsymbol{#1}} 

\newcommand{\thetab}{{\mbs{\theta}}}

\newcommand{\mc}[1]{\mathcal{#1}} 
\newcommand{\Rc}{{\mc{R}}}

\newcommand{\Nc}{{\mc{N}}}
\newcommand{\Fc}{{\mc{F}}}

\newcommand{\Hc}{{\mc{H}}}

\newcommand{\cM}{{\mc{M}}}
\newcommand{\Mcc}{{\mc{M}}}

\newcommand{\Xc}{{\mc{X}}}
\newcommand{\Bc}{{\mc{B}}}

\def\text#1{\hbox{#1}}
\def\proof{{\noindent \bf Proof. }}

\def\endproof{\mbox{\ $\qed$}}

\def\E{{\bf E}}

\def\K{{\bf K}}
\def\P{{\bf P}}
\def\C{{\bf C}}

\def\D{{\bf D}}
\def\B{{\bf B}}

\def\A{{\bf A}}
\def\U{{\bf U}}
\def\H{{\bf H}}
\def\V{{\bf V}}

\def\c{{\bf c}}

\def\a{{\bf a}}
\def\e{{\bf e}}

\def\u{{\bf u}}

\def\q{{q}}

\def\d{\mathrm{d}}
\def\build #1_#2{\mathrel{\mathop{\kern 0pt #1}\limits_\zs{#2}}}
\newcommand{\zs}[1]{{\mathchoice{#1}{#1}{\lower.25ex\hbox{$\scriptstyle#1$}}
{\lower0.25ex\hbox{$\scriptscriptstyle#1$}}}}

\numberwithin{equation}{section}
\def\proof{{\noindent \bf Proof. }}
\def\endproof{\mbox{\ $\qed$}}

\usepackage{mathrsfs}
\newcommand{\Pb}{{\mathsf{P}}} 
\newcommand{\EV}{{\mathsf{E}}} 
\newcommand{\PFA}{\mathsf{PFA}}
\newcommand{\LCPFA}{\mathsf{LCPFA}}
\newcommand{\ADD}{{\mathsf{ADD}}} 
\newcommand{\Hyp}{{\mathsf{H}}} 

\newcommand{\mrm}[1]{\mathrm{#1}}

\newcommand{\F}{{\mrm{F}}}

\newcommand{\mbb}[1]{\mathbb{#1}} 
\def\One{\mathchoice{\rm 1\mskip-4.2mu l}{\rm 1\mskip-4.2mu l}
{\rm 1\mskip-4.6mu l}{\rm 1\mskip-5.2mu l}}
\newcommand\Ind[1]{{\One_{\{#1\}}}} 

\newcommand{\Zbb}{\mbb{Z}} 

\newcommand{\xra}{\xrightarrow} 

\newcommand{\abs}[1]{\left\vert#1\right\vert}
\newcommand{\set}[1]{\left\{#1\right\}}

\newcommand{\brc}[1]{\left(#1\right)}
\newcommand{\brcs}[1]{\left[#1\right]}

\renewcommand{\le}{\leqslant} 
\renewcommand{\ge}{\geqslant}

\usepackage{color}

\begin{document}

\begin{frontmatter}

\title{Asymptotically optimal pointwise and minimax change-point detection for general stochastic models with a composite post-change hypothesis}


\author[A1]{Serguei Pergamenchtchikov}

\author[A2]{Alexander G. Tartakovsky\corref{mycorrespondingauthor}}

\address[A1]{Laboratoire de Math\'ematiques Rapha\"el Salem, 
UMR 6085 CNRS-Universit\'e de Rouen Normandie, France and
National Research Tomsk State University, 
International Laboratory of Statistics of Stochastic Processes and Quantitative Finance, Tomsk, Russia}

\address[A2]{Space Informatics Laboratory, Moscow Institute of Physics and Technology, Moscow, Russia
and AGT StatConsult, Los Angeles, California, USA}

\cortext[mycorrespondingauthor]{Corresponding author. Email address: \url{agt@phystech.edu}}

\begin{abstract}
A weighted Shiryaev-Roberts change detection procedure is shown to approximately minimize the expected delay to detection as well as higher moments of the detection delay among all 
change-point detection procedures with the given low maximal local probability of a false alarm within a window of a fixed length in pointwise and minimax settings for general non-i.i.d.\ data models 
and for the composite post-change hypothesis when the post-change parameter is unknown. 
We establish very general conditions for models under which  the weighted Shiryaev--Roberts procedure is asymptotically optimal. These conditions are formulated in terms of the rate of 
convergence in the strong law of large numbers for the log-likelihood ratios between the ``change'' and ``no-change'' hypotheses, and we also provide sufficient conditions for a  large 
class of ergodic Markov processes. Examples related to multivariate Markov models where these conditions hold are given. 
\end{abstract}


 
\end{frontmatter}

%
\section{Introduction and basic notation} \label{sec:intro}

A substantial part of the development of quickest (sequential) change-point detection has been  directed  towards establishing optimality and asymptotic optimality of certain detection procedures
such as CUSUM, Shiryaev, Shiryaev--Roberts, EWMA and their mutual comparison in various settings (Bayesian, minimax, etc.). See, e.g., 
\citep{BaronTartakovskySA06, Basseville&Nikiforov-book93, Fuh03, girshick-ams52, hawkins-book98, LaiIEEE98, lorden-ams71, masson-book01, MoustakidesAS86, MoustPolTarCS09, 
PergTarSISP2016, Polunchenkoetal-SLJAS2013, PollakTartakovsky-SS09,PolunTartakovskyAS10, ShiryaevTPA63, srivastava-as93, 
TartakovskyIEEEIT2017,TartakovskyIEEEIT2018,TNB_book2014, tartakovskypolpolunch-tpa11, TartakovskyVeerTVP05}. This article is concerned with the problem of minimizing the moments of the detection delay,
$\Rc^r_{\nu,\theta}(\tau)=  \EV_{\nu,\theta}\left[(\tau-\nu)^r\,|\,\tau> \nu \right]$, in pointwise (i.e., for all change points $\nu$) and minimax (i.e., for a worst change point) settings among all procedures
for which the probability of a false alarm $\Pb_{\infty}(k\le \tau < k+m | \tau\ge k)$ is fixed and small. Hereafter $\tau$ is a detection procedure (stopping time), $\nu$ is a point of change, and $\theta$ is a 
post-change parameter.

To be more specific, observations $X_1,X_2, \dots$ are random variables on a probability space $(\Omega, \Fc)$, which may change statistical properties at an unknown point in time 
$\nu\in\{0,1, 2, \dots\}$. For a fixed change point $\nu=k$ and a parameter $\theta\in \Theta$, there is a measure $\Pb_{k,\theta}$ on this space, which in the case of no change ($\nu=\infty$) 
 will be denoted by $\Pb_\infty$. Let  $\EV_{k,\theta}$ denote the expectation under $\Pb_{k,\theta}$ when $\nu=k<\infty$, and let ~$\EV_\infty$
denote the same when there is no change, i.e., $\nu=\infty$. We use the convention that $X_\nu$ is the {\em last pre-change} observation. 
Write $\Xb^{n}=(X_\zs{1},\dots,X_{n})$ for the concatenation of the first $n$ observations. Joint probability densities of $\Xb^n$ are taken with respect to a $\sigma$-finite measure and denoted by 
$p_{k,\theta}(\Xb^n)= p(\Xb^n|\nu=k, \theta)$  when the change point $\nu=k$ is fixed and finite (i.e., joint post-change density) and $p_\infty(\Xb^n)= p(\Xb^n|\nu=\infty)$ stands for the pre-change joint density 
(when the change never occurs). 
Let $\{\psi_{n}(X_\zs{n}|\Xb^{n-1})\}_{n\ge 1}$ and $\{f_{\theta,n}(X_\zs{n}|\Xb^{n-1})\}_{n\ge 1}$ be two sequences of conditional densities of $X_n$ 
given $\Xb^{n-1}$  with respect to some non-degenerate $\sigma$-finite measure. We are interested in the general non-i.i.d.\ case that
\begin{equation}\label{noniidmodel}
\begin{split}
p_{\nu,\theta}(\Xb^n) & = p_\infty(\Xb^n) = \prod_{i=1}^n \psi_i(X_i|\Xb^{i-1}) \quad \text{for}~~ \nu \ge n ,
\\
p_{\nu, \theta}(\Xb^n) & =  \prod_{i=1}^{\nu}  \psi_i(X_i|\Xb^{i-1}) \times \prod_{i=\nu+1}^{n}  f_{\theta,i}(X_i|\Xb^{i-1})  \quad \text{for}~~ \nu < n .
\end{split}
\end{equation}
In other words,  $\{\psi_{n}(X_{n}|\Xb^{n-1})\}_{n\ge 1}$ and $\{f_{\theta,n}(X_{n}|\Xb^{n-1})\}_{n\ge 1}$ are the conditional pre-change and 
post-change  densities, 
respectively, so that if the change occurs
at time $\nu=k$, then the conditional density of the $(k+1)$-th observation changes from $\psi_{k+1}(X_{k+1}|\Xb^{k})$ to $f_{\theta,k+1}(X_{k+1}|\Xb^{k})$. Note that the post-change densities may depend on 
the change point~$\nu$, i.e., $f_{\theta,n}(X_n|\Xb^{n-1})= f_{\theta,n}^{(\nu)}(X_n|\Xb^{n-1})$ for $n > \nu$. We omit the superscript $\nu$ for brevity. 

In many applications, the pre-change distribution is known, but the parameter $\theta$ of the post-change distribution is unknown. In this case,
 the post-change hypothesis ``$\Hyp_k^\vartheta: \nu=k, \theta=\vartheta$'', $\vartheta \in \Theta$ is composite. 

 Obviously, the general non-i.i.d. model given by \eqref{noniidmodel} implies that under the measure~$\Pb_\infty$  
the conditional density of~$X_n$ given $\Xb^{n-1}$
is $\psi_n(X_n|\Xb^{n-1})$ for all $n \ge 1$ and under~$\Pb_{k,\theta}$, for any $0\le k<\infty$, the conditional density of~$X_n$ is
$\psi_n(X_n|\Xb^{n-1})$  if $n \le k$ and is $f_{\theta,n}(X_n|\Xb^{n-1})$ if $n > k$. 

A sequential detection procedure is a stopping (Markov) time $\tau$ for an observed sequence $\{X_\zs{n}\}_{n\ge 1}$, i.e.,  $\tau$ is an
extended integer-valued random variable, such that the event $\{\tau \le n\}$ belongs to the $\sigma$-algebra $\Fc_\zs{n}=\sigma(X_\zs{1},\dots,X_\zs{n})$.  We denote by $\Mcc$ the set of all stopping times.
A false alarm is raised whenever the detection is declared before the change occurs, i.e.,  when $\tau\le \nu$. (Recall that $X_{\nu+1}$ is the first post-change observation.)
The goal of the quickest change-point detection problem is to develop a detection procedure that guarantees a stochastically small delay to detection 
$\tau-\nu$ provided that there is no false alarm (i.e., $\tau> \nu$) under a given (typically low) risk of false alarms.

The present paper extends the results of Pergamenchtchikov and Tartakovsky~\citep{PergTarSISP2016} to the case of the composite post-change hypothesis when the parameter $\theta$ is unknown. 
To this end, we need to develop crucially new synthesis and analysis methods compared to the case of known post-change distribution.
Specifically, we show that the mixture version of the Shiryaev--Roberts procedure (referred in this paper as the weighted SR procedure), which is a natural generalization of the 
Shiryaev--Roberts (SR) procedure in the case of the composite hypothesis,  is asymptotically optimal in the class of procedures with the prescribed 
maximal conditional probability of false alarm when it is small, minimizing moments of the detection delay pointwise (for all change points) as well as in the minimax sense 
(for the worst change point and the worst parameter value). 
While basic ideas and methods are similar to those used in the recent publications by Pergamenchtchikov and Tartakovsky~\citep{PergTarSISP2016} and Tartakovsky~\cite{TartakovskyIEEEIT2018}
the results related to asymptotic optimality by no means trivially follow from these works, especially in the minimax setting. Also, since verification of a general condition that guarantees asymptotic optimality 
related to the $r$-complete-type convergence of the properly normalized log-likelihood ratio to a finite number locally in the vicinity of a true parameter value is a difficult task, an
important contribution is obtaining a set of sufficient conditions for optimality for a wide class of Markov processes (see Section~\ref{sec:Mrk}).

The rest of the paper is organized as follows. In Section~\ref{sec:PF}, we specify the weighted (mixture) Shiryaev--Roberts (WSR) procedure and formulate pointwise and minimax optimization problems of interest.
In Section~\ref{sec:Bay}, we consider the Bayesian version of the problem in the class of procedures with the given weighted probability of false alarm. Based on the recent results of 
Tartakovsky~\citep{TartakovskyIEEEIT2018} we establish asymptotic pointwise and minimax properties of the WSR procedure. These results allow us to establish the main theoretical results in Section~\ref{sec:MaRe}
regarding asymptotic optimality in the class of procedures with the local false alarm probability constraint (in a fixed window). In Section~\ref{sec:Mrk}, we find certain sufficient conditions 
for asymptotic optimality for the class of ergodic Markov processes. In Section~\ref{sec:Ex}, we provide examples of detecting changes in multivariate Markov models. In Section~\ref{sec:MC}, we present
the results of Monte Carlo simulations for a particular example of detecting a change in the correlation coefficient of the Markov Gaussian process that allow us to compare the performance of the weighted SR 
with that of the SR procedure as well as to verify the precision of the first-order asymptotic approximations that follow from the theory.

\section{Problem formulation and the detection procedure} \label{sec:PF}

Let $W$ be a probability distribution on $\Theta$, which is assumed to be quite arbitrary satisfying the condition

\vspace{3mm}
\noindent $(\C_{W})$ {\em For any $\delta>0$, the distribution $W(u)$  is positive on  $\{u \in\Theta\,:\,\vert u-\theta\vert<\delta\}$ 
for every $\theta\in\Theta$: 
\[
W\{u\in\Theta\,:\,\vert u-\theta\vert<\delta \}>0 \quad \forall ~ \delta >0, ~ \theta \in\Theta.
\]
}
\noindent This condition means that we do not consider parameter values of $\theta$ from $\Theta$ of the measure null. See Remark~1 in Section~\ref{sec:Rem} for a discussion of 
the choice of the distribution function $W(\theta)$.

Define the weighted (average) likelihood ratio (LR)
\begin{equation*}
\Lambda_n^k(W) = \int_\Theta \prod_{i=k+1}^{n}  \frac{f_{\theta,i}(X_i|\Xb^{i-1})}{\psi_i(X_i|\Xb^{i-1})} \, \d W(\theta), \quad n >k
\end{equation*}
and the weighted SR statistic
\begin{equation}\label{WSR_stat}
R_n^W = \sum_{k=1}^n \Lambda_n^k(W) \equiv \int_\Theta R_n(\theta) \, {\rm d} W(\theta), \quad n \ge 1, ~~ R_0^W=0,
\end{equation}
where
\begin{equation}\label{SR_stat_noniid}
R_n(\theta) = \sum_{k=1}^n \prod_{i=k}^n \frac{f_{\theta,i}(X_i|\Xb^{i-1})}{\psi_i(X_i|\Xb^{i-1})}
\end{equation}
is the SR statistic tuned to $\theta\in\Theta$.   In this paper, we consider the {\em Weighted Shiryaev--Roberts} (WSR)  detection procedure (or mixture SR) given by the stopping time
\begin{equation}\label{WSR_def2}
T_a =\inf\set{n \ge 1: \log R_n^W \ge a} ,
\end{equation}
where $a>-\infty$ is a threshold controlling for the false alarm risk. In definitions of stopping times we always set $\inf\{\varnothing\}=+\infty$.

For any $0<\beta<1$,  $m\ge 1$, and $\ell \ge 1$ introduce the class of change detection procedures that upper-bounds
the {\em local conditional probability of false alarm} (LCPFA) $\Pb_{\infty}(\tau < k+m | \tau\ge k)= \Pb_{\infty}(k\le \tau < k+m | \tau\ge k)$ in the time interval $[k, k+m-1]$ of the length $m$: 
\begin{equation}\label{sec:Prbf.4}
\Hc(\beta, \ell, m)=\left\{\tau\in\cM: \sup_{1\le k\le \ell}\, \Pb_{\infty}(\tau < k+m | \tau \ge k) \le \beta\right\} ,
\end{equation} 
where $\cM$ is a class of all Markov times. For $r \ge 1$ and $\theta\in\Theta$, we consider the risk associated with the {\em conditional $r$-th moment of the detection delay}
\begin{equation}\label{SCrADD}
\Rc^r_{\nu,\theta}(\tau)=  \EV_{\nu,\theta}\left[(\tau-\nu)^r\,|\,\tau> \nu \right] 
\end{equation}
and the following three problems (pointwise and minimax).

\vspace{3mm}

\noindent {\sl Pointwise  Optimization Problem}:  
\begin{equation}\label{sec:Prbf.5-0}
\inf_{\tau\in \Hc(\beta, \ell, m)}\, \Rc^r_{\nu,\theta}(\tau) \quad \text{for every}~ \nu\ge 0,  ~ \theta\in\Theta .
\end{equation}

\vspace{3mm}

\noindent {\sl Minimax Optimization Problem}:  
\begin{equation}\label{sec:Prbf.5}
\inf_{\tau\in \Hc(\beta, \ell, m)}\,\sup_{0 \le \nu < \infty}\,\Rc^r_{\nu, \theta}(\tau) \quad \text{for every}~ \theta\in\Theta.
\end{equation}

\vspace{3mm}

\noindent  {\sl Double Minimax Optimization Problem}:  
\begin{equation}\label{Minimax2}
\inf_{\tau\in \Hc(\beta, \ell, m)}\,\sup_{\theta\in\Theta}  \, \sup_{0 \le \nu < \infty}\,I_\theta^r \, \Rc^r_{\nu, \theta}(\tau) .
\end{equation}
The function $I(\theta)=I_\theta$, which characterizes  the distance between pre- and post-change distributions, and the parameters $\ell$ and $m$  will be specified later.

However, solving the optimization problems \eqref{sec:Prbf.5-0}--\eqref{Minimax2} for any LCPFA $\beta<1$ is practically impossible, especially for the general non-i.i.d.\ model \eqref{noniidmodel}. 
For this reason, we will focus on the asymptotic problems when the LCPFA $\beta$ goes to $0$. Our goal is to show that the WSR detection procedure $T_a$ with a suitable threshold $a=a_\beta$ is 
first-order asymptotically pointwise and minimax optimal in class $\Hc(\beta, \ell_\beta, m_\beta)$, where $\ell_\beta$ and $m_\beta$ tend to infinity as $\beta\to0$ with a suitable rate.  

In addition, in the next section, we consider a Bayesian-type problem of minimizing the risk \eqref{SCrADD} in a class of procedures with the given weighted probability of false alarm. The solution of this problem is 
used for obtaining the main optimality results in Section~\ref{sec:MaRe} in asymptotic versions of optimization problems \eqref{sec:Prbf.5-0}--\eqref{Minimax2} as $\beta\to0$.


\section{Asymptotic optimality under weighted PFA constraint}\label{sec:Bay} 

\subsection{The non-i.i.d. case}

In order to solve asymptotic (as $\beta\to0$) optimization problems \eqref{sec:Prbf.5-0}, \eqref{sec:Prbf.5} and \eqref{Minimax2}  it is constructive to consider a Bayesian-type class of change detection 
procedures that upper-bounds a weighted probability of false alarm 
\[
\PFA(\tau) = \sum_{k=0}^\infty \Pb_{k,\theta}(\tau \le k) \Pb(\nu =k) = \sum_{k=0}^\infty \Pb_\infty(\tau \le k) \Pb(\nu =k)  , 
\]
assuming that the change point $\nu$ is a random variable independent of the observations with prior distribution 
$\Pb(\nu=k)$, $k\in \Zbb_+ :=\{0,1,2,\dots\}$, and the optimization problems 
\begin{equation}\label{sec:PrbfBayes1}
\inf_{\{\tau: \PFA(\tau) \le \alpha\}}\,\Rc_{k,\theta}^r(\tau),   \quad \forall ~ k \in \Zbb_+ , ~\theta\in \Theta \, ,
\end{equation}
\begin{equation}\label{sec:PrbfBayes}
\inf_{\{\tau: \PFA(\tau) \le \alpha\}}\,\sup_{k\ge 0} \, \Rc_{k,\theta}^r(\tau),   \quad \forall ~\theta\in \Theta \, ,  
\end{equation}
\begin{equation}\label{sec:PrbfBayes3}
\inf_{\{\tau: \PFA(\tau) \le \alpha\}}\, \sup_{\theta\in\Theta}  \, \sup_{k \ge 0}\,I_\theta^r \, \Rc^r_{k, \theta}(\tau) \, ,
\end{equation}
where $0<\alpha <1$ is a prespecified (usually relatively small) number.

This idea was used in \citep{PergTarSISP2016} for establishing asymptotic optimality of the SR procedure in the case of a simple post-change hypothesis.

In what follows, for our purposes it suffices to assume that the prior probability distribution $\Pb(\nu=k)$ of the change point $\nu$ is geometric with the parameter $0<\varrho <1$, i.e.,  for $k \in \Zbb_+$,
\begin{equation}\label{sec:Prbf.1}
\Pb(\nu=k)=\pi_{k}(\varrho)=\varrho\,\left(1-\varrho\right)^{k}\,, 
\end{equation}
so that $\PFA(\tau)= \sum_{k=0}^\infty \varrho\,\left(1-\varrho\right)^{k} \Pb_\infty(\tau \le k)$.

Now, for some fixed $0<\varrho,\alpha<1$, define the following {\em Bayesian class} of change-point detection procedures with the weighted PFA 
not greater that the given number $\alpha$:
\begin{equation}\label{sec:Prbf.2}
\Delta(\alpha,\varrho)= \left\{\tau\in\Mcc:\, \PFA(\tau) \le \alpha\right \}.
\end{equation}

For a fixed $\theta \in \Theta$, introduce the log-likelihood ratio (LLR) process $\{Z_{n}^{k}(\theta)\}_{n \ge k+1}$ between the hypotheses $\Hyp_k^\theta$ ($k\in \Zbb_+$) and $\Hyp_\infty$:
\begin{equation}\label{Znk_df}
Z_{n}^{k} (\theta) = \sum_{j=k+1}^{n}\, \log \frac{f_{\theta,j}(X_{j}|\Xb^{j-1})}{\psi_{j}(X_{j}|\Xb^{j-1})} .
\end{equation}

Assume that there is a positive and finite number $I_\theta$ such that the normalized LLR $n^{-1}Z_{n+k}^k(\theta)$ converges to $I_\theta$ $r$-completely, i.e.,
\begin{equation}\label{rcompletedef}
\sum_{n=1}^\infty n^{r-1} \Pb_{k,\theta}\set{\abs{n^{-1}Z_{n+k}^k(\theta) - I_\theta} > \varepsilon} < \infty, \quad \forall ~ \varepsilon > 0.
\end{equation}
Then it follows from \citep{TartakovskyIEEEIT2017} that in the Bayesian setting, when one wants to minimize the $r$-th moment of the delay to detection \eqref{sec:PrbfBayes1}
and when the parameter $\theta$ is known, the asymptotically (as $\alpha\to 0$ and $\varrho=\varrho_\alpha \to 0$) optimal detection procedure in  class \eqref{sec:Prbf.2} is the 
Shiryaev--Roberts detection procedure that raises an alarm at the first time such that the SR statistic $R_n(\theta)$
 exceeds threshold $(1-\varrho)/\varrho \alpha$. This result was extended by Tartakovsky~\citep{TartakovskyIEEEIT2018} to the case where $\theta$ is unknown. 
 Below we will use the results obtained in \citep{TartakovskyIEEEIT2018}  to show that the WSR procedure \eqref{WSR_def2} is asymptotically optimal in problems \eqref{sec:PrbfBayes1}--\eqref{sec:PrbfBayes3}
under condition \eqref{rcompletedef} and some other conditions. These Bayesian-type asymptotic optimality results are the key for proving asymptotic optimality properties in class $\Hc(\beta,\ell,m)$ with the
constraint on the maximal LCPFA defined in \eqref{sec:Prbf.4}.

 The following condition is sufficient for obtaining the asymptotic  lower bounds for all positive moments of the detection delay in class $\Delta(\alpha)=\Delta(\alpha,\varrho_\alpha)$: \vspace{3mm}

\noindent $(\A_{1}$) {\em  Assume that there exists a positive and finite number $I_\theta>0$ such that for any $k\ge 0$ and $\varepsilon >0$}
\begin{equation}\label{sec:Pmax}
\lim_{N\to\infty} \Pb_{k,\theta}\set{\frac{1}{N}\max_{1 \le n \le N} Z_{k+n}^{k}(\theta) \ge (1+\varepsilon) I_\theta} =0,  \quad \theta\in \Theta .
\end{equation}
Indeed, by Lemma~1 in \citep{TartakovskyIEEEIT2018}, we have that if condition $(\A_1)$ is satisfied and $\varrho_\alpha\to0$ as $\alpha\to0$, then for every $\nu \ge 0$, $\theta\in\Theta$, and $r\ge 1$ 
\begin{equation} \label{LBRnu}
\liminf_{\alpha\to 0} \frac{1}{|\log\alpha|^r}\, \inf_{\tau\in \Delta(\alpha)} \, \sup_{\nu\ge 0} \, \Rc_{\nu,\theta}^r(\tau) \ge  
\liminf_{\alpha\to 0}  \frac{1}{|\log\alpha|^r}\, \inf_{\tau\in \Delta(\alpha)} \, \Rc^r_{\nu,\theta}(\tau) \ge \frac{1}{I_\theta^r} .
\end{equation}

Note that condition $(\A_{1}$)  holds whenever $Z_n^k(\theta)/(n-k)$ converges almost surely to $I_\theta$ under $\Pb_{k,\theta}$:
\begin{equation}\label{sec:MaRe.1}
\frac{1}{n}Z_{k+n}^{k}(\theta) \xra[n\to\infty]{\Pb_{k,\theta}-\text{a.s.}} I_\theta,  \quad  \theta\in \Theta.
\end{equation}
\noindent This is always true for i.i.d. data models with 
\[
I_\theta=\EV_{0,\theta} [Z_{1}^0(\theta)] = \int \log \brcs{\frac{f_\theta(x)}{\psi(x)}} f_\theta(x) \d\mu(x)
\]
being the Kullback--Leibler information number.

The first question we ask is how to select the threshold in the WSR procedure to imbed it into class $\Delta(\alpha,\varrho)$. The following lemma answers this question. 
\begin{lemma}\label{Lem:PFASR} 
For all $a >0$ and any prior distribution of $\nu$ with finite mean $\bar\nu = \sum_{j=1}^\infty j \, \Pb(\nu=j)$, the PFA of the WSR procedure $T_{a}$ given by \eqref{WSR_def2} satisfies the inequality 
\[
\PFA(T_a) \le \bar\nu e^{-a},
\]
so that if
\[
a= a_\alpha = \log \brc{\bar\nu/\alpha}
\]
then $\PFA(T_{a_\alpha}) \le \alpha$.  Therefore, if the prior distribution of the change point is geometric, then $T_{a_\alpha}\in  \Delta(\alpha,\varrho)$ for any $0<\alpha, \varrho<1$ 
if $a_\alpha=\log[(1-\varrho)/\varrho\alpha]$. 
\end{lemma}

\proof
Note that under $\Pb_\infty$ the  WSR statistic $R_n^W$ is a  submartingale with mean $\EV_\infty [R_n^W]=n$. Thus, by Doob's submartingale inequality, for $j\in \{1,2,\dots\}$
 \begin{equation}\label{upperPrj}
 \Pb_\infty(T_a \le j) =\Pb_\infty\brc{\max_{1\le i \le j} R_i^W \ge e^a} \le j \,  e^{-a}
 \end{equation}
and $\Pb_\infty(T_a \le 0) =0$. Hence, for any prior distribution (not necessarily geometric)
\begin{equation*}
\PFA(T_{a})  =\sum_{j=1}^\infty \Pb(\nu=j) \Pb_\infty(T_a \le j) \le \bar{\nu} e^{-a}  .
\end{equation*} 
Therefore, assuming that $\bar{\nu}<\infty$, we obtain that setting 
$a=a_\alpha = \log(\bar\nu/\alpha)$ implies  $\PFA(T_{a_\alpha}) \le \alpha$.  If particularly,  the prior distribution is  geometric, then
$T_{a_\alpha}\in  \Delta(\alpha,\varrho)$ for any $0<\alpha, \varrho<1$ when $a_\alpha=\log[(1-\varrho)/\varrho\alpha]$. 
\endproof

In order to study asymptotic approximations to risks of the WSR procedure and for establishing its asymptotic optimality, we impose the following left-tail condition:\vspace{3mm}

\noindent $(\A_{2}(r)$) {\em There exists a positive continuous $\Theta\to\bbr$ function $I(\theta)=I_\theta$ such that for every $\theta\in\Theta$, for any  $\vae>0$, and for some $r\ge 1$
\begin{equation*}
\Upsilon_r(\varepsilon,\theta):= \lim_{\delta\to 0}\, \sum_{n=1}^\infty \, n^{r-1} \, \sup_{\nu \ge 0} \Pb_{\nu,\theta}\brc{\frac{1}{n} \inf_{\vert u-\theta\vert<\delta} Z_{\nu+n}^\nu(u) < I_\theta  - \varepsilon} <\infty \, .
\end{equation*}
}

To check this condition it is sufficient to check the following condition:

\vspace{3mm}

\noindent $(\A^{*}_{2}(r)$) {\em 
There exists a positive continuous $\Theta\to\bbr$ function $I(\theta)$ such that for every compact set $\K\subseteq \Theta$, for any  $\vae>0$, and for some $r\ge 1$
\begin{equation*}
\Upsilon^{*}_r(\varepsilon, \K):=\,\sup_{\theta\in\K}\,\Upsilon_r(\varepsilon,\theta)<\infty\,.
\end{equation*}
}

In what follows, we assume that the parameter $\varrho$ is a function of $\alpha$ such that
\begin{equation}\label{sec:Bay.6}
\lim_{\alpha\to 0}\,\varrho_{\alpha} =0, \quad\quad \lim_{\alpha\to 0}\,\frac{|\log\varrho_{\alpha}|}{|\log\alpha|}=0 .
\end{equation}

\noindent Moreover, let  $k^{*}$ be a function of $\alpha$ such that
\begin{equation}\label{sec:Bay.7}
\lim_{\alpha\to 0} k^{*}_\alpha=\infty, \quad\quad \lim_{\alpha\to0} \alpha \, \varrho_\alpha \,  k^*_\alpha =0. 
\end{equation}

Note that if in the WSR procedure defined in \eqref{WSR_stat}--\eqref{WSR_def2} the threshold $a= a(\alpha,\varrho_\alpha)$ is selected as 
\[
a_\alpha= (1-\varrho_\alpha)/(\varrho_\alpha \alpha) ,
\] 
where $\varrho_\alpha$ satisfies condition \eqref{sec:Bay.6},  then $a_\alpha \sim |\log\alpha|$ and, by Lemma~\ref{Lem:PFASR}, 
$\PFA(T_{a_\alpha}) \le \alpha$, i.e., this choice of the threshold guarantees that $T_{a_\alpha} \in \Delta(\alpha, \varrho_\alpha) = \Delta(\alpha)$ for every $0<\alpha<1$. 
Hereafter, notation $b_c\sim \tilde{b}_c$ as $c\to c_0$ means that $\lim_{c\to c_0} (b_c/\tilde{b}_c)=1$, i.e., $b_c=\tilde{b}_c(1+o(1))$, where $o(1)\to0$ as $c\to c_0$.

The following theorem establishes first-order asymptotic optimality of the WSR procedure in class $\Delta(\alpha,\varrho_\alpha)=\Delta(\alpha)$.  

\begin{theorem}\label{Th:Bayesopt} 
 Assume that right-tail and left-tail conditions $(\A_{1})$ and $(\A_{2}(r))$ hold  for some  $0<I_\theta<\infty$ and the parameter $0<\varrho=\varrho_\alpha<1$  of the geometric prior distribution 
satisfies conditions \eqref{sec:Bay.6}. 

\noindent {\bf (i)} Then, for all $\theta\in\Theta$ and all fixed $\nu \in \Zbb_+$, as $a\to\infty$, 
\begin{equation}\label{ASapproxBayes}
\Rc_{\nu,\theta}^r(T_a) ~ \sim ~  \brc{\frac{a}{I_\theta}}^r  .
\end{equation}
 If $a=a_\alpha$ is so selected that $\PFA(T_{a_\alpha}) \le \alpha$ and $\log a_\alpha\sim |\log \alpha|$ as $\alpha\to 0$,
in particular $a_\alpha = \log[(1-\varrho_\alpha)/\varrho_\alpha\alpha]$, then, for all $\theta\in\Theta$ and all fixed $\nu \in \Zbb_+$, as $\alpha\to0$, 
\begin{equation}\label{ASoptBayes1}
\inf_{\tau\in\Delta(\alpha)}\Rc_{\nu,\theta}^r(\tau)  ~ \sim ~ \brc{\frac{|\log \alpha|}{I_\theta}}^r ~ \sim ~ \Rc_{\nu,\theta}^r(T_{a_\alpha}) .
\end{equation}

\noindent {\bf (ii)} If $k^*=k^*_\alpha$ satisfies conditions \eqref{sec:Bay.7}  and if $a=a_\alpha$ is so selected that $\PFA(T_{a_\alpha}) \le \alpha$ and $\log a_\alpha\sim |\log \alpha|$ as $\alpha\to 0$,
in particular as $a_\alpha = \log[(1-\varrho_\alpha)/\varrho_\alpha\alpha]$, then, for all $\theta\in\Theta$, as $\alpha\to0$, 
\begin{equation}\label{ASoptBayes2}
\inf_{\tau\in\Delta(\alpha)} \max_{0\le \nu \le k_\alpha^*} \Rc^r_{\nu,\theta}(\tau)  ~ \sim ~ \brc{\frac{|\log \alpha|}{I_\theta}}^r ~ \sim ~  \max_{0\le \nu \le k_\alpha^*} \Rc^r_{\nu,\theta}(T_{a_\alpha}).
\end{equation}

Thus, the WSR procedure $T_{a_\alpha}$ is first-order asymptotically pointwise optimal and minimax in class $\Delta(\alpha,\varrho_\alpha)=\Delta(\alpha)$ with respect to the moments of the 
detection delay up to order $r$.
\end{theorem}

\proof
Asymptotic approximation \eqref{ASapproxBayes} in assertion (i) follows from Theorem~3 in \citep{TartakovskyIEEEIT2018} and  asymptotic approximations \eqref{ASoptBayes1} 
from Theorem~4 in \citep{TartakovskyIEEEIT2018}.

It remains to prove asymptotic approximations \eqref{ASoptBayes2} in assertion (ii).  If $a_\alpha = \log[(1-\varrho_\alpha)/\varrho_\alpha\alpha]$, then using inequality  \eqref{upperPrj} we obtain 
that for any $0 \le \nu \le k^*_\alpha$,
\begin{equation}\label{PFALB}
\Pb_{\nu,\theta}\left(T_{a_\alpha} > \nu\right)= \Pb_{\infty}\left(T_{a_\alpha} > \nu\right)\ge  \Pb_{\infty}\left(T_{a_\alpha} > k_\alpha^{*} \right) \ge 1-  k^*_\alpha \, e^{-a_\alpha} = 
1- \frac{\alpha \varrho_\alpha k_\alpha^*}{1-\varrho_\alpha} .
\end{equation}
The second condition in \eqref{sec:Bay.7} implies that $\Pb_{\infty}\left(T_{a_\alpha} > \nu\right) \to 1$ as $\alpha\to0$  for all $0 \le \nu\le k^*_\alpha$. Since 
$\Rc_{\nu,\theta}^r(T_{a_\alpha}) =\EV_{\nu,\theta}[(T_{a_\alpha} -\nu)^+]^r/\Pb_{\infty}\left(T_{a_\alpha} > \nu\right)$,
inequality \eqref{LBRnu} along with equalities \eqref{ASoptBayes1} imply \eqref{ASoptBayes2} for $k^*=k^*_\alpha$ satisfying conditions \eqref{sec:Bay.7}.
This completes the proof.
\endproof

\begin{remark}
While for the sake of simplicity we consider the geometric prior distribution with the small parameter 
$\varrho_\alpha$, all the asymptotic results hold true for an arbitrary prior distribution 
$\pi_k^\alpha$ such that  the mean value of the change point $\EV[\nu] = \sum_{k=1}^\infty k \pi_k^\alpha$ approaches infinity as $\alpha\to 0$,
assuming that conditions \eqref{sec:Bay.6} and \eqref{sec:Bay.7} hold with $\varrho_\alpha$ replaced by $(\sum_{k=1}^\infty k \pi_k^\alpha)^{-1}$.  
\end{remark}

It is also interesting to ask whether the WSR procedure is asymptotically optimal with respect to the following double minimax criterion
\[
  \sup_{\theta\in\Theta}  \max_{0\le \nu \le k_\alpha^*} I_\theta^r \, \Rc_{\nu,\theta}^r(\tau) \longrightarrow  ~~ \text{minimum over}~ \tau\in\Delta(\alpha), ~~ \alpha\to0
\]
 (see \eqref{sec:PrbfBayes3}). The following theorem gives an affirmative answer for compact subsets $\Theta_1 \subset \Theta$.

\begin{theorem}\label{Th:MinMax}
 Assume that the right-tail condition $(\A_{1})$ is satisfied  for some  $0<I_\theta<\infty$,  the parameter $0<\varrho=\varrho_\alpha<1$  of the geometric prior distribution 
satisfies conditions \eqref{sec:Bay.6}, and  conditions \eqref{sec:Bay.7} hold for $k^*=k^*_\alpha$. Assume that for every $\vae>0$ and some $r\ge 1$
\begin{equation}\label{rLeft}
\sup_{\theta \in \Theta_1} \Upsilon_r(\varepsilon,\theta)<\infty,  \quad \quad \inf_{\theta\in \Theta_1} I_\theta >0,
\end{equation}
where $\Theta_1$ is compact.  If $a=a_\alpha$ is so selected that $\PFA(T_{a_\alpha}) \le \alpha$ and $\log a_\alpha\sim |\log \alpha|$ as $\alpha\to 0$,
in particular as $a_\alpha = \log[(1-\varrho_\alpha)/\varrho_\alpha\alpha]$, then 
\begin{equation}
  \inf_{\tau\in\Delta(\alpha)} \sup_{\theta\in\Theta_1}  \max_{0\le \nu \le k_\alpha^*} I_\theta^r \, \Rc_{\nu,\theta}^r(\tau) \sim |\log\alpha|^r \sim   
  \sup_{\theta\in\Theta_1}  \max_{0\le \nu \le k_\alpha^*} I_\theta^r \, \Rc_{\nu,\theta}^r(T_{a_\alpha}), \quad \alpha\to 0. 
\end{equation}
\end{theorem}

\proof
Using inequalities \eqref{LBRnu}, we obtain that under condition $(\A_{1})$  the following asymptotic lower bound holds: 
\begin{equation} \label{LowerMinMax} 
\inf_{\tau\in \Delta(\alpha)} \, \sup_{\theta \in \Theta_1} I_\theta^r \, \Rc^r_{\nu,\theta}(\tau) \ge |\log \alpha|^r(1+o(1)) , \quad \alpha \to 0 .
\end{equation}

To prove the theorem it suffices to show that the right-hand side in \eqref{LowerMinMax} is attained for the risk 
\[
\sup_{\theta\in\Theta_1}  \max_{0\le \nu \le k_\alpha^*} I_\theta^r \Rc_{\nu,\theta}^r(T_{a_\alpha})
\]
 of the WSR procedure. Using inequality (A.25) in \citep{TartakovskyIEEEIT2018}, we obtain that for an arbitrary $0<\varepsilon < I_\theta$
\begin{equation*}
\sup_{\theta\in\Theta_1} I_\theta^r \sup_{\nu\ge 0} \EV_{\nu,\theta}\brcs{(T_a-\nu)^+}^r \le  
\sup_{\theta\in\Theta_1} I_\theta^r \brc{1+\frac{a}{I_\theta-\varepsilon}}^r +  r2^{r-1} \,  \sup_{\theta\in\Theta_1} I_\theta^r \,\Upsilon_r(\varepsilon,\theta).
 \end{equation*}
By condition \eqref{rLeft}, the second term on the right side is finite, which immediately implies that 
\[
\sup_{\theta\in\Theta_1} I_\theta^r \sup_{\nu\ge 0} \EV_{\nu,\theta}\brcs{(T_a-\nu)^+}^r \le a^r (1+o(1)), \quad  a \to \infty.
\]
Next, using inequality \eqref{PFALB}, we obtain that for $ \nu \le k^*_\alpha$
\[
\Rc_{\nu,\theta}^r(T_{a_\alpha}) = \frac{\EV_{\nu,\theta}[(T_{a_\alpha} -\nu)^+]^r}{\Pb_{\infty}\left(T_{a_\alpha} > \nu\right)} \le \frac{\EV_{\nu,\theta}[(T_{a_\alpha} -\nu)^+]^r}{1-  k^*_\alpha e^{-a_\alpha}}
= \frac{(1-\varrho_\alpha) \EV_{\nu,\theta}[(T_{a_\alpha} -\nu)^+]^r}{1-  \varrho_\alpha - \alpha \varrho_\alpha k^*_\alpha},
\]
which along with the previous inequality and conditions \eqref{sec:Bay.6}--\eqref{sec:Bay.7} yields the upper bound
\[
\sup_{\theta\in\Theta_1}  \max_{0\le \nu \le k_\alpha^*} I_\theta^r \, \Rc_{\nu,\theta}^r(T_{a_\alpha}) \le  |\log \alpha|^r(1+o(1)), \quad \alpha \to 0
\]
and the proof is complete.
\endproof

\begin{remark}
Let 
\[
T_B^*(\theta) = \inf\set{n\ge 1: R_n(\theta) \ge B}
\]
be the stopping time of the SR detection procedure tuned to $\theta$, where $R_n(\theta)$ is the SR statistic defined in \eqref{SR_stat_noniid}. If the least favorable value of the parameter $\theta_*$ that maximizes 
the risk
\[ 
\sup_{\theta\in\Theta_1}  \max_{0\le \nu \le k_\alpha^*} I_\theta^r \Rc_{\nu,\theta}^r(T^*_{B_\alpha}(\theta)) = \max_{0\le \nu \le k_\alpha^*} I_{\theta_*}^r \Rc_{\nu,\theta_*}^r(T^*_{B_\alpha}(\theta_*))
\]
can be found (at least approximately within a small term), then the SR rule $T_B^*(\theta_*)$ is asymptotically double minimax. This rule is easier to implement and it should have even smaller maximal risk
than the WSR rule.
\end{remark}

\subsection{The case of LLR with independent increments}\label{ssec:LLRindincr}

We now show that condition $(\A_2(r))$ can be substantially relaxed in the case where observations are independent, but not necessarily identically distributed, i.e., 
$f_{\theta,i}(X_i|\Xb^{i-1})= f_{\theta, i}(X_i)$ and $\psi_i(X_i|\Xb^{i-1})=\psi_i(X_i)$ in \eqref{noniidmodel}. More generally, we may assume that the increments 
$\Delta Z_i(\theta)= \log [f_{\theta,i}(X_i|\Xb^{i-1})/\psi_i(X_i|\Xb^{i-1})]$ of the LLR 
$Z_{n}^k(\theta)=\sum_{i=k+1}^n \Delta Z_i(\theta)$ are independent, which is always the case if the observations are independent. 
This slight generalization is important for certain examples with dependent observations that lead to the LLR with independent increments. 

\begin{theorem}\label{Th:FOAOindep}
 Assume that the LLR process $\{Z_{k+n}^k(\theta)\}_{n\ge 1}$ has independent, 
not necessarily identically distributed increments under $\Pb_{k,\theta}$, $k\ge 0$.  Suppose that condition \eqref{sec:Pmax} holds and the following condition is satisfied 
\begin{equation}\label{probLeft}
\lim_{\delta\to 0} \lim_{n\to\infty}  \Pb_{\nu,\theta}\brc{\frac{1}{n} \int_{\Gamma_\delta} Z_{\ell+n}^\ell(\vartheta) \, \mrm{d} W(\vartheta) < I_\theta  - \varepsilon} =0, \quad \varepsilon >0, ~
 \ell \ge \nu, ~ \theta\in \Theta.
\end{equation}
If the parameter $\varrho=\varrho_\alpha$ of the geometric prior distribution goes to zero 
as $\alpha\to 0$ at rate defined in \eqref{sec:Bay.6}, then relations \eqref{ASoptBayes1} and \eqref{ASoptBayes2} hold for all $r>0$, i.e., 
the WSR procedure is asymptotically optimal  with respect to all positive moments of the detection delay.
\end{theorem}

\begin{proof}
Let $N=N(a,\varepsilon,\theta)= 1+\lfloor a/(I_\theta-\varepsilon) \rfloor$.  Hereafter $\lfloor x \rfloor$ denotes the integer number less than or equal to $x$.
We begin with showing that the asymptotic upper bound 
\begin{equation}\label{UpperRra}
\Rc^r_{\nu,\theta}(T_{a}) \le \brc{\frac{a}{I_\theta}}^r (1+o(1)), \quad  a \to \infty
\end{equation}
holds for all $r\ge 1$ under condition \eqref{probLeft}. To this end, note that we have the following chain of equalities and inequalities:
\begin{align} \label{ExpkTAplus}
\EV_{\nu,\theta}\brcs{(T_a-\nu)^+}^r   & =  \sum_{\ell=0}^{\infty} \int_{\ell N}^{(\ell+1) N} r t^{r-1}  \Pb_{\nu,\theta} (T_a-k > t) \, \mrm{d} t   
 \le N^r + \sum_{\ell =1}^{\infty} \int_{\ell N}^{(\ell +1) N} r t^{r-1}  \Pb_{\nu,\theta} (T_a-k > t) \, \mrm{d} t  \nonumber
\\
& \le  N^r + \sum_{\ell=1}^{\infty} \int_{\ell N}^{(\ell+1) N} r t^{r-1}  \Pb_{\nu,\theta} (T_a-k > \ell N) \, \mrm{d} t  \nonumber
 = N^r\brc{1 + \sum_{\ell=1}^{\infty} [(\ell+1)^r-\ell^r]  \Pb_{\nu,\theta} (T_a-k > \ell N )} \nonumber
\\
&  \le  N^{r} \brc{1+\sum_{\ell=1}^{\infty}   r (\ell+1)^{r-1}   \Pb_{\nu,\theta} (T_a -k>  \ell N )} 
 \le  N^{r} \brc{1+ r 2^{r-1} \sum_{\ell=1}^{\infty}  \ell^{r-1}   \Pb_{\nu,\theta} (T_a -k>  \ell N )}.  
\end{align}

Consider the intervals (cycles) of the length $N$ and let $K_n(\nu,N)=K_n= \nu+n N$ and
\[
\lambda_{\nu+ n N}^{\nu+(n-1)N}(W) := \int_{\Gamma_\delta} Z_{\nu+ n N}^{\nu+(n-1)N}(\vartheta) \, \mrm{d} W(\vartheta) =
 \sum_{i=K_{n-1}+1}^{K_n}\int_{\Gamma_\delta} \Delta Z_i(\vartheta) \, \mrm{d} W(\vartheta)\,.
\]
Since for any $n\ge 1$ 
\[
\log R_{\nu+n N}^W \ge  \log \Lambda_{\nu+ n N}^{\nu+(n-1)N}(W)  \ge \log \int_{\Gamma_\delta} \exp\set{Z_{\nu+ n N}^{\nu+(n-1)N}(\vartheta)} \, \mrm{d} W(\vartheta)
\]
and, by Jensen's inequality,
\[
 \int_{\Gamma_\delta} \exp\set{Z_{\nu+ n N}^{\nu+(n-1)N}(\vartheta)} \, \mrm{d} W(\vartheta) = 
  W(\Gamma_\delta) \int_{\Gamma_\delta} \exp\set{Z_{\nu+ n N}^{\nu+(n-1)N}(\vartheta)} \, \frac{\mrm{d} W(\vartheta)}{W(\Gamma_\delta)} \ge W(\Gamma_\delta)
  \exp\set{\frac{\lambda_{\nu+ n N}^{\nu+(n-1)N}(W)}{W(\Gamma_\delta)}},
\]
it follows that
\[
 \log R_{\nu+n N}^W \ge \frac{\lambda_{K_n}^{K_{n-1}} (W)}{W(\Gamma_\delta)}  - |\log W(\Gamma_\delta)| , \quad n \ge 1.
\]
Therefore,
\begin{align*}
\Pb_{\nu,\theta} \brc{T_a-\nu > \ell N}  & =\Pb_{\nu,\theta}\brc{\log R_{n}^W < a ~\text{for}~ n\in\{1,\dots,\nu+\ell N\}} 
 \le  \Pb_{\nu,\theta}\brc{\log R_{\nu+n N}^W < a ~\text{for}~ n\in\{1,\dots,\ell\}} 
\\
&
 \le   \Pb_{\nu,\theta}\brc{\frac{\lambda_{K_n}^{K_{n-1}} (W)}{W(\Gamma_\delta)}  < a + |\log W(\Gamma_\delta)| ~\text{for}~ n\in\{1,\dots,\ell\}}\,.
\end{align*}
Since the increments of the LLR are independent, the random variables $\lambda_{K_n}^{K_{n-1}}(W)$, $n\in\{1,\dots,\ell\}$,  are independent, and hence,
\begin{align*}
\Pb_{\nu,\theta} \brc{T_a-\nu > \ell N}  & \le     \Pb_{\nu,\theta}\brc{\frac{\lambda_{K_n}^{K_{n-1}} (W)}{W(\Gamma_\delta)}  < a + |\log W(\Gamma_\delta)| ~\text{for}~ n\in\{1,\dots,\ell\}} 
= \prod_{n=1}^\ell \Pb_{\nu,\theta}\brc{\frac{\lambda_{K_n}^{K_{n-1}} (W)}{W(\Gamma_\delta)} < a + |\log W(\Gamma_\delta)|}
\\
&\le \prod_{n=1}^\ell \Pb_{\nu,\theta}\brc{\frac{\lambda_{K_n}^{K_{n-1}} (W)}{N W(\Gamma_\delta)}  <I_\theta-\varepsilon + |\log W(\Gamma_\delta)|/N}\,.  
\end{align*}
It follows that for a sufficiently large $a>0$, for which $ |\log W(\Gamma_\delta)|/N\le \varepsilon/2$,
$$
\Pb_{\nu,\theta} \brc{T_a-\nu > \ell N} \le \prod_{n=1}^\ell \Pb_{\nu,\theta}\brc{\frac{\lambda_{K_n}^{K_{n-1}} (W)}{N W(\Gamma_\delta)}  <I_\theta-\varepsilon/2}\,. 
$$
By condition \eqref{probLeft}, for a sufficiently large $a$ there exists a small $\delta_a$ such that 
\[
\Pb_{\nu,\theta}\brc{\frac{1}{NW(\Gamma_\delta) }\int_{\Gamma_\delta} Z_{K_n}^{K_{n-1}}(\vartheta) \, \mrm{d}W(\vartheta) < I_\theta  - \varepsilon/2} \le \delta_a, \quad n \ge 1.
\]
Therefore, for any $\ell \ge 1$, 
\[
\Pb_{\nu,\theta} \brc{T_a-\nu > \ell N}  \le \delta_a^\ell.
\]
Combining this inequality with \eqref{ExpkTAplus} and using the fact that $L_{r,a}=\sum_{\ell=1}^\infty \ell^{r-1} \delta_a^\ell \to 0$ as $a\to\infty$ for any $r>0$, we obtain
\begin{align}\label{Momineq}
\Rc^r_{\nu,\theta}(T_{a_\alpha})  &= \frac{\EV_{\nu,\theta}\brcs{(T_{a_\alpha}-\nu)^+}^r}{\Pb_\infty(T_a>\nu)} 
 \le \frac{\brc{1+\frac{\log a}{I_\theta-\varepsilon}}^r + r 2^{r-1}\, L_{r,a}}{1-\nu e^{-a}}   
 = \brc{\frac{a}{I_\theta-\varepsilon}}^r(1+o(1)),  \quad a\to\infty.
\end{align}
Since $\varepsilon\in(0,I_\theta)$ is an arbitrary number, this implies the upper bound \eqref{UpperRra}.

Setting $a_\alpha = \log[(1-\varrho_\alpha)/\varrho_\alpha\alpha]$ or more generally $a_\alpha\sim|\log\alpha|$ in \eqref{Momineq} yields the upper bound (for all $r \ge 1$,  all $\nu\ge 0$, and all $\theta\in\Theta$)
\begin{equation*}
\Rc^r_{\nu,\theta}(T_{a_\alpha}) \le \brc{\frac{|\log\alpha|}{I_\theta}}^r (1+o(1)), \quad \alpha\to0 .
\end{equation*}
Applying this upper bound together with the lower bound \eqref{LBRnu} (which holds due to condition \eqref{sec:Pmax}) proves \eqref{ASoptBayes1}.

The proof of \eqref{ASoptBayes2} is essentially analogous to that in the proof of Theorem~\ref{Th:Bayesopt}(ii) above. 
\end{proof}

\section{Asymptotic optimality under local PFA constraint} \label{sec:MaRe}

We now proceed with the pointwise and minimax problems \eqref{sec:Prbf.5-0}, \eqref{sec:Prbf.5} and \eqref{Minimax2} in the class of procedures with given LCPFA $\Hc(\beta,\ell,m)$
defined in \eqref{sec:Prbf.4}. Note that the asymptotic optimality results of the previous section are essential, 
since asymptotic optimality in class $\Hc(\beta,\ell,m)$ is obtained by imbedding this class in class $\Delta(\alpha,\varrho)$ with specially selected parameters $\varrho$ and $\alpha$.

\subsection{The non-i.i.d.\ case}

For any  $0<\beta<1$, $m \ge 1$, and $\ell \ge 1$, define
\begin{equation}\label{sec:Absrsk.1-0}
\alpha_{1}= \alpha_{1}(\beta,m) =\beta+(1-\varrho_{1,\beta})^{m+1}
\end{equation}
and
\begin{equation}\label{sec:Cnrsk.1}
\alpha_{2}= \alpha_{2}(\beta,\ell,m)= \frac{\beta(1-\varrho_{2,\beta})^{\ell+m}} {1+\beta},
\end{equation}
where $\varrho_{2,\beta}=\check{\delta}_{\beta}\,\varrho_{1,\beta}$ and the functions $0<\varrho_{1,\beta}<1$ and $0<\check{\delta}_{\beta} <1$ are such that
\begin{equation}\label{sec:Absrsk.1-01}
\lim_{\beta\to 0}\,\left( \varrho_{1,\beta}+\check{\delta}_{\beta}\right)=0,
\quad\quad \lim_{\beta\to 0}\,\frac{\vert\log\varrho_{1,\beta}\vert+\vert\log\check{\delta}_{\beta}\vert}{\vert \log\beta\vert}=0\,.
\end{equation}
For example, we can take
\begin{equation*}
\varrho_{1,\beta}=\frac{1}{1+|\log \beta|}, \quad \check{\delta}_{\beta}= \frac{\check{\delta}^{*}}{1+|\log \beta|}\quad\mbox{with}\quad  0<\check{\delta}^{*}<1 .
\end{equation*}

To find asymptotic lower bounds for the problems  \eqref{sec:Prbf.5-0}  and \eqref{sec:Prbf.5} in addition to condition $(\A_{1})$ we impose the following condition related to the growth of the window 
size $m$ in the LCPFA: \vspace{3mm}

\noindent $(\H_1)$ {\em  The size of the window $m$ in \eqref{sec:Absrsk.1-0}  is a function of $\beta$, i.e. $m=m_{\beta}$,  such that
\begin{equation}\label{sec:Absrsk.4}
\lim_{\beta\to 0}\, \frac{|\log \alpha_{1,\beta}|}{|\log\beta|}=1 ,
\end{equation}
where  $\alpha_{1,\beta}=\alpha_{1}(\beta,m_{\beta})$.}

The following theorem establishes asymptotic lower bounds. 

\begin{theorem} \label{Th.sec:Cnrsk.1} 
 Assume  that  conditions $(\A_{1})$ and $(\H_{1})$ hold. Then, for any $\ell \ge 1$,  $\nu\in \Zbb_+$, $\theta\in\Theta$, and $r\ge 1$ 
\begin{align} \label{sec:Cnrsk.5}
\liminf_{\beta\to 0} \frac{1}{|\log\beta|^r} \inf_{\tau\in \Hc(\beta,\ell,m_\beta)}\, \sup_{\nu\ge 0} \, \Rc^r_{\nu,\theta}(\tau)\,  \ge 
\liminf_{\beta\to 0} \frac{1}{|\log\beta|^r} \inf_{\tau\in \Hc(\beta,\ell,m_\beta)}\, \Rc^r_{\nu,\theta}(\tau) 
\ge \frac{1}{I_\theta^r} .
\end{align}
\end{theorem}

\proof
First, recall that under condition $(\A_1)$ the lower bounds \eqref{LBRnu} hold for any $r\ge1$ and $\nu\ge 0$.
Second, we show that for any $0<\beta<1$,  $m \ge |\log (1-\beta)|/[|\log (1-\varrho_{1,\beta})|]- 1:=m_0$ and $\ell \ge 1$, the following inclusion holds:  
\begin{equation}\label{Incl1}
\Hc(\beta,\ell,m) \subseteq \Delta(\alpha_{1},\varrho_{1,\beta}) ,
\end{equation}
where $\alpha_1=\alpha_1(\beta,m) < 1$ for every $m > m_0$ and $\beta\in(0,1)$. Indeed, let $\tau$ be from $\Hc(\beta, \ell,m)$. 
Then, using definition of class $\Hc(\beta, \ell,m)$, we obtain that $\Pb_{\infty}(\tau \le m) \le \beta$. 
Therefore,  taking in \eqref{sec:Prbf.1} $\varrho=\varrho_{1,\beta}$, we obtain
\begin{align*}
\sum_{k= 0}^\infty \,\pi_{k}(\varrho_{1,\beta})\, \Pb_{\infty}\left(\tau \le k\right) &= \sum^{m}_{k=0}\,\pi_{k}(\varrho_{1,\beta})\,
\Pb_{\infty}\left(\tau \le  k\right)+ \sum^{\infty}_{k=m+1}\,\pi_{k}(\varrho_{1,\beta})\, \Pb_{\infty}\left(\tau \le k\right)
\\
&\le \beta+\sum^{\infty}_{k= m+1}\,\pi_{k}(\varrho_{1,\beta}) = \beta+\left(1-\varrho_{1,\beta}\right)^{m+1} =\alpha_{1},
\end{align*}
i.e., $\tau\in \Delta(\alpha_{1},\varrho_{1,\beta})$ where $\alpha_1<1$ for $m >m_0$.

Inclusion \eqref{Incl1} implies that for all $\nu\ge 0$  and for a sufficiently small $\beta>0$
 $$
\inf_{\tau\in \Hc(\beta,\ell,m)}\,\Rc^r_{\nu,\theta}(\tau)\,\ge \inf_{\tau\in \Delta(\alpha_{1,\beta},\varrho_{1,\beta})}\,\Rc^r_{\nu,\theta}(\tau).
$$ 
Now, lower bounds \eqref{sec:Cnrsk.5} follow from lower bounds \eqref{LBRnu} and  condition $(\H_{1})$. 
\endproof

To establish asymptotic optimality properties of the WSR procedure with respect to the risks $\Rc_{\nu,\theta}^r(\tau)$ (for all $\nu\ge 0$) and $\sup_{\nu\ge0}\Rc_{\nu,\theta}^r(\tau)$ 
in class $\Hc(\beta,\ell,m)$ we need the uniform  $r$-complete convergence condition $(\A_\zs{2}(r))$ as well as  the following condition: \vspace{3mm}

\noindent $(\H_{2})$  {\em 
Parameters $\ell$ and $m$ are functions of $\beta$, i.e. $\ell=\ell_{\beta}$ and $m=m_{\beta}$, such that
\begin{equation}\label{sec:Cnrsk.6}
\lim_{\beta\to 0}\frac{ |\log\alpha_{2,\beta}|}{|\log\beta|}=1, \quad\quad \lim_{\beta\to0}  \alpha_{2,\beta} \, \varrho_{2,\beta}\,  k_\beta^* =0,
\end{equation}
where $\alpha_{2,\beta}=\alpha_{2}(\beta,\ell_{\beta},m_\beta)$ and $k_\beta^*=\ell_\beta+m_\beta$.} 

The conditions  \eqref{sec:Absrsk.4} and \eqref{sec:Cnrsk.6}  hold, for example, if
\begin{equation*}
m_{\beta}=\lfloor \vert\log\beta\vert/\varrho_{1,\beta}\rfloor , \quad\quad \ell_{\beta}=\check{\varkappa}\,m_{\beta}, \quad \check{\varkappa} >0.
\end{equation*}

Denote by $T_\beta$ the WSR procedure $T_{a_\beta}$  with threshold $a_\beta$ given by 
\begin{equation}\label{sec:Cnrsk.7}
 a_{\beta}= \log \brc{\frac{1-\alpha_{2,\beta}}{\varrho_{2,\beta}\alpha_{2,\beta}}}. 
\end{equation}

\begin{theorem}\label{Th.sec:Cnrsk.2} 
If conditions  $(\H_{1})$ and $(\H_{2})$ hold, then, for any $0<\beta<1$, the WSR procedure $T_\beta$  with threshold $a_\beta$ given by \eqref{sec:Cnrsk.7} belongs to 
class $\Hc(\beta,\ell_\beta,m_\beta)$. Assume in addition that conditions $(\A_{1})$ and $(\A_{2}(r))$ are satisfied. Then 
\begin{equation}\label{sec:SRAO2}
  \inf_{\tau\in\Hc(\beta,\ell_\beta,m_\beta)}\, \Rc_{\nu,\theta}^r(\tau)~ \sim~ \brc{\frac{|\log \beta|}{I_\theta}}^r ~ \sim ~  \Rc_{\nu,\theta}^r(T_{\beta}),  \quad  \nu\in \Zbb_+, ~\theta\in\Theta
\end{equation}
and
\begin{equation}\label{sec:SRAO2b}
 \inf_{\tau\in\Hc(\beta,\ell_\beta,m_\beta)}\, \max_{0\le \nu\le k^{*}_\beta} \Rc_{\nu,\theta}^r(\tau)~ \sim~ \brc{\frac{|\log \beta|}{I_\theta}}^r  ~ \sim ~
 \max_{0\le \nu\le k^{*}_\beta} \Rc_{\nu,\theta}^r(T_{\beta}), \quad\theta\in\Theta.
\end{equation}
Therefore, the WSR procedure $T_\beta$ is first-order  asymptotically pointwise optimal and minimax in class $\Hc(\beta,\ell_\beta,m_\beta)$, minimizing moments of the detection delay
up to order $r$ for all parameter values $\theta\in\Theta$.
\end{theorem}

\proof
First note that the conditions \eqref{sec:Absrsk.1-01} and \eqref{sec:Cnrsk.6} imply the properties \eqref{sec:Bay.6}--\eqref{sec:Bay.7}. 
Now, by Lemma~\ref{Lem:PFASR}, the WSR procedure $T_{a_\alpha(\varrho)}$ with threshold $a_\alpha(\varrho)=\log[(1-\alpha)/\varrho\alpha]$ belongs to class $\Delta(\alpha,\varrho)$ 
for any $0<\alpha, \varrho<1$. Hence, $T_{\beta}\in \Delta(\alpha_{2,\beta},\varrho_{2,\beta})$.  Note that for any $0<\beta<1$,  $m \ge 1$ and $\ell \ge 1$, 
the following inclusion holds:  
\begin{equation}\label{sec:Cnrsk.2}
\Delta(\alpha_{2,\beta},\varrho_\zs{2,\beta}) \subseteq \Hc(\beta,\ell,m) .
\end{equation}
Indeed, by definition of class $\Delta(\alpha,\varrho)$, we have that for any $0<\alpha, \varrho<1$ and any $i \ge 1$
\[
\alpha \ge \varrho \sum_{k=i}^\infty (1-\varrho)^k \Pb_\infty(\tau \le k) \ge \varrho \Pb_\infty(\tau \le i) \sum_{k=i}^\infty (1-\varrho)^k = \Pb_\infty(\tau \le i) (1-\varrho)^i,
\]
which implies that
\[
\sup_{\tau \in \Delta(\alpha,\varrho)} \Pb_\infty(\tau \le k^{*}_\beta)  \le \alpha (1-\varrho)^{-k^{*}_\beta}.
\] 
Let $\tau\in \Delta(\alpha_{2,\beta},\varrho_{2,\beta})$. Then, taking into account the latter inequality and using definition of $\alpha_{2,\beta}$ in \eqref{sec:Cnrsk.1}, we obtain that
\begin{align*}
\sup_{1\le k\le \ell_\beta}\, \Pb_{\infty}(\tau < k+m_\beta | \tau \ge k)&\le \sup_{1\le k\le \ell_\beta}\,\frac{\Pb_{\infty}\left(\tau < k+m_\beta\right)} {\Pb_{\infty}\left( \tau \ge k\right)}
 \le \frac{\Pb_{\infty}\left(\tau < k^{*}_\beta\right)} {1-\Pb_{\infty}\left( \tau < k^{*}_\beta\right)}
\le \frac{\alpha_{2,\beta}\,(1-\varrho_{2,\beta})^{-k^{*}_\beta}} {1-\alpha_{2,\beta}\,(1-\varrho_{2,\beta})^{-k^{*}_\beta}} =\beta,
\end{align*}
i.e., $\tau$ belongs to $\Hc(\beta,\ell_\beta,m_\beta)$. 

Using inclusion~\eqref{sec:Cnrsk.2}, we obtain that  the stopping time $T_{\beta}$ belongs to $\Hc(\beta,\ell_\beta,m_\beta)$ for any $0<\beta<1$. 

Next, in view of definition of $a_{\beta}$ in \eqref{sec:Cnrsk.7} and of the form of the function $\varrho_{2,\beta}$ in \eqref{sec:Absrsk.1-01}
we obtain, using condition $(\H_{2}$), that $\lim_{\beta\to 0}\, a_{\beta}/|\log\beta|=1$. Thus, by \eqref{ASapproxBayes} in Theorem~\ref{Th:Bayesopt},
\[
\lim_{\beta\to\infty}\, \frac{1}{|\log \beta|^r} \, \Rc_{\nu,\theta}^r(T_\beta) = \frac{1}{I_\theta^r}, \quad \nu \in \Zbb_+, ~  \theta\in\Theta.
\]
Comparing to the lower bound \eqref{sec:Cnrsk.5} implies \eqref{sec:SRAO2}.

To prove \eqref{sec:SRAO2b} it suffices to show that
\begin{equation}\label{UBbeta}
\limsup_{\beta\to0} \frac{\max_{0\le \nu \le k^{*}_\beta}\,\Rc_{\nu,\theta}^r(T_{\beta})}{|\log\beta|^r} \le \frac{1}{I^r_\theta},  \quad \theta\in\Theta.
\end{equation}
Note that
$$
\max_{0\le \nu \le k^{*}_\beta}\, \Rc_{\nu,\theta}^r(T_{\beta}) \le 
\dfrac{\max_{0\le \nu \le k^{*}_\beta}\,  \EV_{\nu,\theta}\left[(T_{\beta}-\nu)^+\right]^r}{\min_{0\le \nu \le k^{*}_\beta}\,\Pb_{\infty}\left(T_{\beta}> \nu \right)},
$$
where as $\beta\to0$
\[
\min_{0\le \nu \le k^{*}_\beta}\,\Pb_{\infty}\left(T_{\beta}> \nu \right) = \Pb_{\infty}\left(T_{\beta}> k^*_\beta \right) \to 1 .
\]
Also, by inequality (A.25) in \citep{TartakovskyIEEEIT2018}, for an arbitrary $0<\varepsilon < I_\theta$,
\begin{equation*}
\sup_{\nu\ge 0} \EV_{\nu,\theta}\brcs{(T_\beta-\nu)^+}^r \le  
 \brc{1+\frac{a_\beta}{I_\theta-\varepsilon}}^r +  r2^{r-1} \,  \,\Upsilon_r(\varepsilon,\theta).
 \end{equation*}
Note that the second condition in \eqref{sec:Absrsk.1-01} and the first condition in \eqref{sec:Cnrsk.6} imply that $a_\beta \sim |\log \beta|$ as $\beta\to0$.
As a result, we obtain as $\beta\to0$
\begin{align*}
\max_{0\le \nu \le k^{*}_\beta}\,\Rc_{\nu,\theta}^r(T_{\beta})  \le \dfrac{\sup_{\nu\ge 0 }\,  \EV_{\nu,\theta}\left[(T_{\beta}-\nu)^+]^r\right]}{\Pb_{\infty}\left(T_{\beta}> k^*_\beta \right)}
 = \brc{\frac{|\log \beta|}{I_\theta}}^r(1+o(1)) .
\end{align*}
This obviously yields  the upper bound \eqref{UBbeta} and the proof is complete.
\endproof

The following theorem establishes minimax properties of the WSR procedure with respect to the risk 
\[
\sup_{\theta\in\Theta_1}  \max_{0\le \nu \le k^{*}_\beta} I_\theta^r \Rc_{\nu,\theta}^r(\tau)
\] 
for compact subsets $\Theta_1$ of $\Theta$. It is absolutely similar to Theorem~\ref{Th:MinMax} and its proof follows almost immediately from Theorem~\ref{Th:MinMax} and 
Theorem~\ref{Th.sec:Cnrsk.2}, and for this reason it is omitted. 

\begin{theorem}\label{Th:MinMax2}
 Assume that the right-tail and the left-tail conditions $(\A_{1})$ and \eqref{rLeft} are satisfied  for some  $0<I_\theta<\infty$ and that conditions  $(\H_{1})$ and $(\H_{2})$ are satisfied as well. Then 
\begin{equation*}
  \inf_{\tau\in \Hc(\beta,\ell_\beta,m_\beta)}\,  \sup_{\theta\in\Theta_1}  \max_{0\le \nu\le k^{*}_\beta} I_\theta^r \, \Rc_{\nu,\theta}^r(\tau) \sim |\log\beta|^r \sim   
  \sup_{\theta\in\Theta_1}  \max_{0\le \nu\le k^{*}_\beta}  I_\theta^r \, \Rc_{\nu,\theta}^r(T_{\beta})  \quad \text{as}~ \beta\to0.
\end{equation*}
\end{theorem}

\subsection{The case of LLR with independent increments}\label{ssec:LLRindincr2}

As in Subsection~\ref{ssec:LLRindincr}, we now consider a particular (still quite general) case where the LLR process has independent increments. The following theorem is similar to Theorem~\ref{Th:FOAOindep}.

\begin{theorem}\label{Th:FOAOindep2}
 Assume that the LLR process $\{Z_{k+n}^k(\theta)\}_{n\ge 1}$ has independent, 
not necessarily identically distributed increments under $\Pb_{k,\theta}$, $k \ge 0$.  Suppose that conditions  $(\H_{1})$, $(\H_{2})$, $(\A_1)$, and \eqref{probLeft} are satisfied.  
Then asymptotic relations \eqref{sec:SRAO2} and \eqref{sec:SRAO2b} 
hold for all $r>0$, i.e.,  the WSR procedure $T_\beta$ is first-order asymptotically pointwise optimal and minimax in class $\Hc(\beta,\ell_\beta,m_\beta)$ with respect to all positive moments of the detection delay
for all $\theta\in\Theta$.
\end{theorem}

\begin{proof}
The proof follows almost immediately from Theorem~\ref{Th:FOAOindep} and Theorem~\ref{Th.sec:Cnrsk.2} and is omitted.
\end{proof}

\section{Uniform concentration inequalities and sufficient conditions of asymptotic optimality for Markov processes}
\label{sec:Mrk}
Condition $(\A_1)$ is usually not difficult to check since it follows from the almost sure convergence of the log-likelihood ratio \eqref{sec:MaRe.1}. However, verification of the left-tail $r$-complete-type 
convergence condition $(\A_2(r))$ may be a challenge. In this section, we obtain sufficient conditions for a class of homogeneous Markov processes in order to verify condition $(\A_2(r))$  
in particular examples.

Let  $(X^{\theta}_\zs{n})_\zs{n\ge 1}$ be a time homogeneous  Markov process with values in a measurable space $(\Xc,\Bc)$ defined by a family
of the transition probabilities $(P^{\theta}(x,A))_\zs{\theta\in\Theta}$ for some fixed parameter set $\Theta\subseteq \bbr^{p}$.  In the sequel, we denote by $\E^{\theta}_\zs{x}(\cdot)$   the  expectation 
with respect to this probability. In addition, we assume that this process is geometrically ergodic, i.e., \vspace{3mm}

\noindent $(\B_\zs{1})$
{\em For any $\theta\in\Theta$ there exist a probability measure $\lambda$ on $(\Xc,\Bc)$ and the  Lyapunov $\Xc\to [1,\infty)$
function $\V^{\theta}$ with $\lambda^{\theta}(\V^{\theta})<\infty$,  such that for some positives constants $0<R<\infty$ and $\kappa>0$,
$$
\sup_\zs{n\ge 0}\, e^{\kappa n}\, \sup_\zs{x\in\bbr}\,\sup_\zs{\theta\in\Theta}\,\sup_\zs{0<F \le \V^{\theta}}\,\,
\frac{1}{\V^{\theta}(x)}\,\left|\E^{\theta}_\zs{x}\,[F(X_\zs{n})]-\lambda^{\theta}(F)\right|\le\, R\,.
$$
}
Now, for some $\q>0$, we set
\begin{equation}\label{sec:Mrk.3}
\upsilon^{*}_\zs{\q}(x)=\sup_\zs{n\ge 0}\,\sup_\zs{\theta\in\Theta}\,\E^{\theta}_\zs{x}\,\left[\V^{\theta}(X_\zs{n})\right]^{\q}\,.
\end{equation}

\noindent
Let $g$ be a measurable $\Theta\times\Xc\times\Xc\to\bbr$ function
 such that the following integrals exist
\begin{equation}\label{sec:Mrk.1_Int_g}
\wt{g}(\theta,x)=\int_\zs{\Xc}\,g(\theta,y,x)\,P^{\theta}(x,\d y), \quad\quad \lambda^{\theta}(\wt{g}) =\int_\zs{\Xc}\,\wt{g}(u)\,\lambda^{\theta}(\d u)\,.
\end{equation}

\noindent $(\B_\zs{2})$
{\em
Assume that $g$ satisfies the H\"older condition of the power $0<\gamma\le 1$ with respect to the first variable, i.e., there exists a measurable positive $\Xc\times\Xc\to\bbr$ function $h$
such that for any $x,y$ from $\Xc$
\begin{equation}\label{sec:Mrk.3_hh}
\sup_\zs{u\,,\,\theta\in\Theta}\frac{\vert g(u,y,x)-g(\theta,y,x)\vert}{\vert u-\theta\vert^{\gamma}}\le h(y,x)
\end{equation}
and the corresponding integrals $\wt{h}(\theta,x)$ and $\lambda^{\theta}(\wt{h})$ exist for any $\theta\in\Theta$, where $\wt{h}(\theta,x)$ is defined as $\wt{g}$ in \eqref{sec:Mrk.1_Int_g} and $\vert\cdot\vert$ 
is the Euclidean norm in $\bbr^{p}$.
}
\vspace{3mm}

\noindent $(\B_\zs{3})$
{\em Assume that the functions $g$ and $h$ are such that $|\wt{g}(\theta,x)|\le \V^{\theta}(x)$ and $|\wt{h}(\theta,x)|\le \V^{\theta}(x)$ for all $\theta\in\Theta$ and $x\in\Xc$.}

\vspace{2mm}

Now, for any measurable $\Theta\times\Xc\times\Xc\to\bbr$ function $g$
for which there exist the integrals \eqref{sec:Mrk.1_Int_g}, we introduce the deviation processes
\begin{equation}\label{sec:Mrk.1-1}
W^{g}_\zs{n}(u,\theta)=n^{-1}\sum^{n}_\zs{j=1}\,g(u,X_\zs{j},X_\zs{j-1}) -\lambda^{\theta}(\wt{g}), \quad\quad \wt{W}^{g}_\zs{n}(\theta)=W^{g}_\zs{n}(\theta,\theta)\,.
\end{equation}

Similarly to \eqref{sec:Mrk.3}, we define for some $\q>0$ 
\begin{equation}\label{sec:Mrk.2}
g^{*}_\zs{\q}(x)=\sup_\zs{n\ge 1}\,\sup_\zs{\theta\in\Theta}\,\E^{\theta}_\zs{x}\,|g(\theta,X_\zs{n},X_\zs{n-1})|^{\q} ,
\quad\quad h^{*}_\zs{\q}(x)=\sup_\zs{n\ge 1}\, \sup_\zs{\theta\in\Theta}\,\E^{\theta}_\zs{x}\,|h(X_\zs{n},X_\zs{n-1})|^{\q}\,.
\end{equation}

Proposition 1 of \citep{PergTarSISP2016} implies the following result.

\begin{proposition} \label{Pr.sec:Mrk.1_theta} 
Assume that conditions $(\B_\zs{1})-(\B_\zs{3})$ hold. Then for any   $\q\ge 2$, for which $\upsilon^{*}_\zs{\q}(x)<\infty$, $g^{*}_\zs{\q}(x)<\infty$, and $h^{*}_\zs{\q}(x)<\infty$, one has
\begin{equation}\label{sec:Mrk.1_theta}
\sup_\zs{n\ge 2}\,n^{\q/2}\,\sup_\zs{\theta\in\Theta}\,\sup_\zs{x\in\Xc}\,\frac{\E^{\theta}_\zs{x}|\wt{W}^{g}_\zs{n}(\theta)|^{\q}+\E^{\theta}_\zs{x}|\wt{W}^{h}_\zs{n}(\theta)|^{\q}}
{\left(1+\upsilon^{*}_\zs{\q}(x)+g^{*}_\zs{\q}(x)+h^{*}_\zs{\q}(x)\right)}<\infty\,.
\end{equation}
\end{proposition}

\begin{proposition} \label{Pr.sec:Mrk.2_theta} 
Assume that conditions 
$(\B_\zs{1})-(\B_\zs{3})$ hold and
$\upsilon^{*}_\zs{\q}(x)<\infty$, $g^{*}_\zs{\q}(x)<\infty$, and $h^{*}_\zs{\q}(x)<\infty$ for some  $\q\ge 2$.
 Then for any 
$\varepsilon>0$ there exists $\delta_\zs{0}>0$
such that
\begin{equation}\label{sec:Mrk.4E_ev}
\sup_\zs{0<\delta\le \delta_\zs{0}} \sup_\zs{n\ge 2}\, n^{\q/2} \,\sup_\zs{\theta\in\Theta}\,\sup_\zs{x\in\Xc}\,
\frac{\P^{\theta}_\zs{x}\left(\sup_\zs{\vert u-\theta\vert<\delta}\vert W^{g}_\zs{n}(u,\theta)\vert>\varepsilon \right)}
{\left(1+\upsilon^{*}_\zs{\q}(x)+g^{*}_\zs{\q}(x)+h^{*}_\zs{\q}(x)\right)}<\infty\,.
\end{equation}
\end{proposition}

\proof
First note that
$$
\sup_\zs{\vert u-\theta\vert<\delta}\vert W^{g}_\zs{n}(u,\theta)\vert \le \delta^{\gamma}\left(\lambda^{\theta}(\wt{h})+
\wt{W}^{h}_\zs{n}(\theta)\right)+\vert \wt{W}^{g}_\zs{n}(\theta)\vert\,.
$$
Therefore, for any 
$$
0< \delta\le \left(\frac{\varepsilon}{2(1+\lambda^{\theta}(\wt{h}))}\right)^{1/\gamma}:=\delta_\zs{0}
$$
we obtain that
$$
\P^{\theta}_\zs{x}\left(\sup_\zs{\vert u-\theta\vert<\delta}\vert W^{g}_\zs{n}(u,\theta)\vert>\varepsilon \right)\\
\le  \P^{\theta}_\zs{x}\left(\vert \wt{W}^{g}_\zs{n}(\theta)\vert>\varepsilon/2\right)+\P^{\theta}_\zs{x}\left(\vert \wt{W}^{h}_\zs{n}(\theta)\vert>1\right)\,.
$$
Applying the bound \eqref{sec:Mrk.1_theta} we obtain \eqref{sec:Mrk.4E_ev}. Hence Proposition ~\ref{Pr.sec:Mrk.2_theta} follows.
\endproof

We return to the change detection problem for Markov processes, assuming
that the sequence of observations $(X_\zs{n})_\zs{n\ge 1}$ is a Markov process, such that
$(X_\zs{n})_\zs{1\le n \le \nu}$ is a homogeneous process
 with the transition (from $x$ to $y$) density $\psi(y|x)$. In the sequel, we denote by $\check{\P}$
 the distribution of this process when $\nu=\infty$, i.e., 
 when the Markov process $(X_\zs{j})_\zs{j\ge 1}$ has transition density 
 $\psi(y|x)$. The expectation with respect to this distribution we denote by $\check{\E}$.
Moreover, $(X_\zs{n})_\zs{n> \nu}$ is 
homogeneous positive  ergodic  with the transition density $f_\zs{\theta}(y|x)$
and the ergodic (stationary) distribution $\lambda^{\theta}$. 
The densities $\psi(y|x)$ and  $f_\zs{\theta}(y|x)$   are calculated 
with respect to a sigma-finite positive measure $\mu$ on $\Bc$.

In this case, we can represent the process $Z^{k}_\zs{n}(\theta)$ defined in \eqref{Znk_df} as
\begin{equation*}
Z^{k}_\zs{n}(\theta)=\sum^{n}_\zs{j=k+1} g(\theta,X_\zs{j},X_\zs{j-1})\,, \quad g(\theta,y,x)=\log \frac{f_\zs{\theta}(y|x)}{\psi(y|x)}\,.
\end{equation*}
\noindent
Therefore, in this case,
\begin{equation*}
\wt{g}(\theta,x)=\int_\zs{\Xc}\,g(\theta,y,x)\,f_\zs{\theta}(y|x)\,\mu(\d y) \,.
\end{equation*}
Moreover, if we assume that density $f_\zs{u}(y|x)$ is continuously differentiable with respect to $u$ in a compact set  $\K\subseteq\Theta$, then the inequality \eqref{sec:Mrk.3_hh}
holds with $\gamma=1$ and for any function $h(y,x)$ for which
\begin{equation*}
\sup_\zs{u\in\Theta}\,\max_\zs{1\le j\le p}\,\vert \partial g(u,y,x)/\partial u_\zs{j}\vert \le h(y,x)\,.
\end{equation*}

\noindent $(\C_\zs{1})$ {\em Assume that there exists a set $C\in\Bc$ with $\mu(C)<\infty$ such that

\begin{enumerate}
 
 \item[$({\rm C}1.1)$] $f_\zs{*}=  \inf_\zs{\theta\in\K}\, \inf_\zs{x,y\in C}\,f_\zs{\theta}(y|x)>0$.

\item[$({\rm C}1.2)$] 
For any $\theta\in\K$ there exists  $\Xc\to [1,\infty)$ Lyapunov's function $\V^{\theta}$ such that

\begin{itemize}
\item  $\V^{\theta}(x)\ge \wt{g}(\theta,x)$ and $\V^{\theta}(x)\ge \wt{h}(\theta,x)$ for any $\theta\in\K$ and $x\in\Xc$,

\item $\sup_\zs{\theta\in\Theta}\, \sup_\zs{x\in C} \V^{\theta}(x)<\infty$.

\item For some $0<\rho<1$ and $D>0$ and for all  $x\in\Xc$ and $\theta\in\Theta$, 
\begin{equation}\label{drift_ineq_11}
\E^{\theta}_\zs{x}[\V^{\theta}(X_\zs{1})] \le  (1-\rho) \V^{\theta}(x) + D\Ind{C}(x)\,.
\end{equation}
\end{itemize}
\end{enumerate}
}
 
\noindent 
$(\C_\zs{2})$ 
{\em Assume that there exists $\q> 2$ such that 
$$
\sup_\zs{k\ge 1}\check{\E}\,[g^{*}_\zs{\q}(X_\zs{k})] <\infty \,,\quad \sup_\zs{k\ge 1}\,\check{\E} [h^{*}_\zs{\q}(X_\zs{k})] <\infty ,
\quad\quad \sup_\zs{k\ge 1}\,\check{\E}[\upsilon^{*}_\zs{\q}(X_\zs{k})] <\infty\,,
$$
where the function $\upsilon^{*}_\zs{\q}(x)$ is defined in \eqref{sec:Mrk.3}, $g^{*}_\zs{\q}(x)$ and $h^{*}_\zs{\q}(x)$ are given in \eqref{sec:Mrk.2}.
}

\begin{theorem} \label{Th.sec:Mrk.1} 
Assume that conditions $(\C_\zs{1})-(\C_\zs{2})$ hold for some compact set $\K\subseteq\Theta$. Then for any $0<r<\q/2$ condition $(\A^{*}_\zs{2}(r))$ holds with 
 $I(\theta)=\lambda^{\theta}(\wt{g})$.
\end{theorem}

\proof
First note that it follows from Theorem \ref{Th.AMc.1}  in the appendix that the conditions $(\C_\zs{1})$ imply the property $(\B_\zs{1})$. So Proposition \ref{Pr.sec:Mrk.2_theta} yields
 that there exists a positive constant $\C^{*}$ such that for any $x\in\Xc$
$$
\P^{\theta}\left(\sup_\zs{\vert u-\theta\vert<\delta}\vert W^{g}_\zs{n}(u,\theta)\vert>\varepsilon \vert X_\zs{0}=x\right) \le \C^{*}\,\U^{*}(x)\,n^{-\q/2}\,,
$$
where $\U^{*}(x)=1+\upsilon^{*}_\zs{\q}(x)+g^{*}_\zs{\q}(x)+h^{*}_\zs{\q}(x)$. Note now that
$$
\Pb_{\nu,\theta}\brc{\frac{1}{n} \inf_{\vert u-\theta\vert<\delta} Z_{\nu+n}^\nu(u) < I_\theta  - \varepsilon} 
 \le \Pb_{\nu,\theta}\brc{\sup_{\vert u-\theta\vert<\delta}\,\left\vert\frac{1}{n} Z_{\nu+n}^\nu(u) - I_\theta  \right\vert >\varepsilon} .
$$
In view of the homogeneous Markov property we obtain that for $I_\zs{\theta}=\lambda^{\theta}(\wt{g})$ the last probability can be represented as
$$
 \Pb_{\nu,\theta}\brc{\sup_{\vert u-\theta\vert<\delta}\,\left\vert\frac{1}{n} Z_{\nu+n}^\nu(u) - I_\theta \right\vert >\varepsilon}  = \check{\E}[\Psi^{\theta}(X_\zs{\nu})] \,,
$$
where $\Psi^{\theta}(x)=\P^{\theta}\left(\sup_\zs{\vert u-\theta\vert<\delta}\vert W^{g}_\zs{n}(u,\theta)\vert>\varepsilon \vert X_\zs{0}=x\right)$. Therefore,
$$
\Pb_{\nu,\theta}\brc{\sup_{\vert u-\theta\vert<\delta}\,\left\vert\frac{1}{n} Z_{\nu+n}^\nu(u) - I_\theta  \right\vert>\varepsilon} 
 \le  \C^{*}\,n^{-\q/2}\check{\E}[\U^{*}(X_\zs{\nu})]\,.
$$
Now condition  $(\C_\zs{2})$ implies  $(\A^{*}_\zs{2}(r))$ for any $0<r<\q/2$.
\endproof

Note that  condition $({\rm C}1.1)$ does not always hold for the process $(X_\zs{n})_\zs{n\ge 1}$ directly.
Unfortunately, this condition does not hold for the practically important autoregression process of the order more than one.
For this reason, we need to weaken this requirement. Assume that there exists $p\ge 2$
for which the  process $(\wt{X}_\zs{\iota,n})_\zs{n\ge \wt{\nu}}$ for $\wt{\nu}=\nu/p-\iota$ defined as $\wt{X}_\zs{\iota,n}=X_\zs{n p+\iota}$ satisfies the following properties:

\vspace{2mm}
\noindent $(\C^{\prime}_1)$ {\em Assume that there exists a set $C\in\Bc$ with $\mu(C)<\infty$ such that

\begin{enumerate}
 
 \item[$(\C^{\prime}1.1)$] 
 $\wt{f}_\zs{*}=\inf_\zs{1\le \iota\le p}\,\inf_\zs{\theta\in\K}\, \inf_\zs{x,y\in C}\,\wt{f}_\zs{\iota,\theta}(y|x)>0$,
 where $\wt{f}_\zs{\iota,\theta}(y|x)$ is the transition density for the  process $(\wt{X}_\zs{\iota,n})_\zs{n\ge 1}$.

\item[$(\C^{\prime}1.2)$] 
For any $\theta\in\K$ there exists  $\Xc\to [1,\infty)$ Lyapunov's function $\V^{\theta}$ such that 
\begin{equation}
\label{bound_V}
\upsilon^{*}=\max_\zs{1\le j\le p}\,\sup_\zs{\theta\in\K}\,\sup_\zs{x\in\Xc}\, \frac{\E^{\theta}_\zs{x}\,[\V^{\theta}(X_\zs{j})]}{\V^{\theta}(x)} \,<\,\infty\,;
\end{equation}
\begin{equation}
\label{bound_V11}
\upsilon^{*}_\zs{1}=\sup_\zs{\theta\in\K}\,\lambda^{\theta}(\V^{\theta})\,<\,\infty\,.
\end{equation}

\begin{itemize}
\item 
$\V^{\theta}(x)\ge \wt{g}(\theta,x)$ and $\V^{\theta}(x)\ge \wt{h}(\theta,x)$ for $\theta\in\K$ and $x\in\Xc$ and
$
 \sup_\zs{\theta\in\K}\,\sup_\zs{x\in C} \V^{\theta}(x)<\infty\,. 
$

\item For some $0<\rho<1$ and $D>0$ and for all  $x\in\Xc$, $\theta\in\K$, and $0\le \iota\le p-1$
\begin{equation}\label{drift_in_+++}
\E^{\theta}\left[\V^{\theta}(\wt{X}_\zs{\iota,1})\vert \wt{X}_\zs{\iota,0}=x\right] \le  (1-\rho) \V^{\theta}(x) + D\Ind{C}(x)\,.
\end{equation}
\end{itemize}
\end{enumerate}
}

\begin{theorem} \label{Th.sec:Mrk.1'} 
Assume that conditions $(\C 2)$ and $(\C^\prime_1)$ hold. Then for any $0<r<\q/2$ condition $(\A^{*}_\zs{2}(r))$ holds with $I(\theta)=\lambda^{\theta}(\wt{g})$.
\end{theorem}
\proof
Note again that by Theorem \ref{Th.AMc.1} (see Appendix) conditions $(\C_1)$ yield condition $(\B_1)$ for $(\wt{X}_\zs{\iota,n})_\zs{n\ge \wt{\nu}}$, i.e.,
for some positive constants $0<R_\zs{\iota}<\infty$ and $\kappa_\zs{\iota}>0$,
$$
\sup_\zs{n\ge 0}\, e^{\kappa_\zs{\iota} n}\, \sup_\zs{x\in\bbr} \, \sup_\zs{\theta\in\K} \, \sup_\zs{0< F \le \V^{\theta}}\, \, \frac{1}{\V^{\theta}(x)}\, \left| D^{\theta}_\zs{m}(x) \right| \le\,  R_\zs{\iota}\,,
$$
where $D^{\theta}_\zs{m}(x)=\E^{\theta}\,\left[\left(F(\wt{X}_\zs{\iota,m})-\lambda^{\theta}(F)\right)\vert \wt{X}_\zs{\iota,0}=x\right]$. 
So, for any $n\ge p$ we can write that $n=mp+\iota$ for some $0\le \iota\le p-1$ and we obtain 
$$
\vert \E^{\theta}_\zs{x}(F(X_\zs{n})-\lambda^{\theta}(F)) \vert = \vert \E^{\theta}_\zs{x} \,D^{\theta}_\zs{m}(X_\zs{\iota}) \vert \le  \, R_\zs{\iota}\, \V^{\theta}(X_\zs{\iota}) e^{\kappa_\zs{\iota} m} \,.
$$
Now, the upper bound \eqref{bound_V} implies
$$
\sup_\zs{n\ge p}\, e^{\kappa n}\, \sup_\zs{x\in\bbr} \, \sup_\zs{\theta\in\K} \, \sup_\zs{0< F \le \V^{\theta}}\, \,
\frac{1}{\V^{\theta}(x)}\, \vert \E^{\theta}_\zs{x}[F(X_\zs{n})-\lambda^{\theta}(F)] \vert \, \le\,  R_\zs{*}\,,
$$
where $R_\zs{*}=\upsilon^{*}\max_\zs{0\le \iota}R_\zs{\iota}$ and $\kappa=\min_\zs{0\le\iota\le p-1}\kappa_\zs{\iota}/p$. Thus, using bound \eqref{bound_V11}, we obtain condition $(\B_\zs{1})$ with 
$R=e^{\kappa_\zs{*}p} (\upsilon^{*}+\upsilon^{*}_\zs{1}) + R_\zs{*}$. Using now the same argument as in the proof of Theorem \ref{Th.sec:Mrk.1} we obtain Theorem \ref{Th.sec:Mrk.1'}.
\endproof

\section{Examples}\label{sec:Ex}

We now present examples of detecting changes in multivariate Markov models that illustrate the general theory developed in Sections~\ref{sec:Bay} and \ref{sec:MaRe}. 

\begin{example}{(Change in the parameters of the multivariate linear difference equation).}
Consider the multivariate model in $\bbr^{p}$ given by
\begin{equation*}
X_\zs{n} = \left(\check{A}_\zs{n}\Ind{n\le \nu}+A_\zs{n}\Ind{n>\nu}\right)\,X_\zs{n-1}+w_\zs{n}\,,
\end{equation*}
where $\check{A}_\zs{n}$ and $A_\zs{n}$ are $p\times p$ random matrices and $(w_\zs{n})_\zs{n\ge 1}$ is an i.i.d. sequence of Gaussian random vectors $\Nc(0, Q_\zs{0})$ in
$\bbr^{p}$ with the positive definite $p\times p$ matrix $Q_\zs{0}$. Assume also that $\check{A}_\zs{n}=A_\zs{0}+B_\zs{n}$, $A_\zs{n}=\theta+B_\zs{n}$
and $(B_\zs{n})_\zs{n\ge 1}$ are i.i.d. Gaussian random matrices $\Nc(0\,,Q_\zs{1})$, where the $p^{2}\times p^{2}$ matrix $Q_\zs{1}=\E[B_\zs{1}\otimes B_\zs{1}]$ is  positive definite. 
Hereafter, for $p\times p$  matrices  $\U=(\u_{ij})_{1\le i,j\le p}$ and $\V=(\textbf{v}_{ij})_{1\le i,j\le p}$ the $p^{2}\times p^{2}$ matrix $\U\otimes \V$ is
$$
\U\otimes \V=(\u_{ij}\textbf{v}_{kl})_{1\le i,j,k,l\le p}\,.
$$

Assume, in addition, that all eigenvalues of the matrix 
$
\E[\check{A}_\zs{1}\,\otimes\,\check{A}_\zs{1}] =A_\zs{0}\otimes A_\zs{0}+Q_\zs{1}
$
are less than one in module. Define
\begin{equation*}
\Theta=\{\theta\in\bbr^{p^{2}}\,:\, \max_\zs{1\le j \le p^{4}} \e_\zs{j}(\theta\otimes\theta+Q_\zs{1})<1\}\}\setminus \,\{A_\zs{0}\}\,,
\end{equation*}
where $\e_\zs{j}(A)$ is the $j$th eigenvalue of matrix $A$, and assume further that the matrix $\theta\in\Theta$.
In this case,  the processes  $(X_\zs{n})_\zs{n\ge 1}$ (in the case $\nu=\infty$) and $(X_\zs{n})_\zs{n> \nu}$ (in the case $\nu<\infty$)
are ergodic with the ergodic distributions given by the vectors \cite{KlPe04}
$$
\check{\varsigma}=\sum_\zs{i\ge 1}\prod^{i-1}_\zs{j=1}\check{A}_\zs{j}\,w_\zs{i},
\quad\quad \varsigma_\zs{\theta}=\sum_\zs{i\ge 1}\prod^{i-1}_\zs{j=1}A_\zs{j}\,w_\zs{i} ,
$$
i.e.,  the corresponding invariant measures $\check{\lambda}$ and $\lambda^{\theta}$ on $\bbr^{p}$ are 
defined as $\check{\lambda}(A)=\Pb(\check{\varsigma}\in \Gamma)$ and $\lambda^{\theta}(A)=\Pb(\varsigma_\zs{\theta}\in \Gamma)$
for  any $\Gamma\in\Bc(\bbr^{p})$. According to \citep{FeiginTweedieJTSA85} we define the  Lyapunov function as 
\begin{equation}
 \label{LpF_V__}
 \V^{\theta}(x)=\upsilon_\zs{*}(1+x^\top T(\theta)x)\,, \quad T(\theta)=\left(I_\zs{p^{4}}-\theta^\top\otimes\theta^\top- Q^\top_\zs{1}\right)^{-1} {\rm vec}(\check{I}_\zs{p}) \,,
 \end{equation}
where the symbol $\top$ in $(\cdot)^\top$ denotes transpose, $\upsilon_\zs{*}\ge 1$, $\check{I}_\zs{m}$ is the identity matrix of order $m$, and for the $p\times p$ matrix  $\V=(\textbf{v}_{ij})_{1\le i,j\le p}$ 
the vector ${\rm vec}(\V)$ is
$$
{\rm vec}(\V)=(\textbf{v}_{11},\ldots,\textbf{v}_{p1},\dots,\textbf{v}_{1p},\dots,\textbf{v}_{pp})^\top\in\bbr^{p^2}\,.
$$

As shown in \citep{FeiginTweedieJTSA85}, in this case,  for any $x\in\bbr^{p}$ the quadratic form $x^\top T(\theta)x\ge \vert x\vert^{2}$.  Hence, all eigenvalues of the matrix $T(\theta)$
are greater than one. Let now $\K\subset\Theta$ be some compact set. For some fixed $N_\zs{*}>1$, define the set
 \begin{equation*}
 C=\{x\in\bbr^{p}\,:\, \max_\zs{\theta\in\K}x^\top T(\theta)x\le N_\zs{*}\}\,.
 \end{equation*}
 By direct calculation we obtain 
 \begin{equation*}
 \E^{\theta}\,\left[\V^{\theta}(X_\zs{1}) \vert X_\zs{0}=x\right] =\V^{\theta}(x) \left(1-\frac{\vert x\vert^{2}-\tr T(\theta)Q_\zs{0}}{\V^{\theta}(x)}\right)\,.
 \end{equation*}
Taking into account that the function $T(\theta)$ is continuous, we obtain that for any non-zero vector $x\in\bbr^{p}$ and $\theta\in\K$
$$
1\le \e_\zs{min}\le \frac{x^{\top}T(\theta)x}{\vert x\vert^{2}} \le \e_\zs{max}\,<\infty\,,
$$
where
$$
\e_\zs{min}=\min_\zs{\theta\in\K}\,\inf_\zs{x\in\bbr^{p},x\neq 0}\frac{x^{\top}T(\theta)x}{\vert x\vert^{2}},
\quad\quad \e_\zs{max}=\max_\zs{\theta\in\K}\,\sup_\zs{x\in\bbr^{p},x\neq 0}\frac{x^{\top}T(\theta)x}{\vert x\vert^{2}}\,.
$$ 
It follows that, for $x\in C^{c}$,  $\vert x\vert^{2}>N_\zs{*}/\e_\zs{max}$ and, therefore,
$$
\frac{\vert x\vert^{2}-\tr T(\theta)Q_\zs{0}}{\V^{\theta}(x)}\ge \frac{\vert x\vert^{2}}{1+\e_\zs{max}\vert x\vert^{2}}
-\frac{\e_\zs{max} \tr Q_\zs{0}}{\e_\zs{min}\vert x\vert^{2}}\ge \frac{1}{2\e_\zs{max}}-\frac{\e^{2}_\zs{max} \tr Q_\zs{0}}{\e_\zs{min}\,N_\zs{*}}\,.
$$ 
 Now we choose $N_\zs{*}>1$ sufficiently large to obtain positive term in the right side of the last inequality. So we obtain 
 the drift inequality \eqref{drift_ineq_11} for the Lyapunov function defined in  \eqref{LpF_V__} with  any coefficient $\upsilon^{*}\ge 1$. The function $g(u,y,x)$ can be calculated for any 
 $x,y \in \bbr^{p}$ and $u\in\Theta$ as
\begin{align*}
g(u,y,x) =\frac{\vert G^{-1/2}(x)(y-A_\zs{0}x)\vert^{2}-\vert G^{-1/2}(x)(y-u\,x)\vert^{2}}{2}=y^{\top}G^{-1}(x)(u-A_\zs{0})x+\frac{x^{\top}A^{\top}_\zs{0}G^{-1}(x)\,A_\zs{0}x-x^{\top}u^{\top}\,G^{-1}(x)\,u\,x}{2}\,,
\end{align*}
where $G(x)=\E\, [B_\zs{1}xx^{\top}B^{\top}_\zs{1}]+Q_\zs{0}=Q_\zs{1}\vect(xx^{\top})+Q_\zs{0}$. Taking into account that the matrices $Q_\zs{0}$ and
$Q_\zs{1}$ are positive definite, we obtain that there exists some constant $\c_\zs{*}>0$ for which   
\begin{equation}\label{sec:Ex.6-0n.1}
\sup_\zs{x\in\bbr^{p}}\,\vert G^{-1}(x)\vert\le \,\frac{\c_\zs{*}}{1+\vert x\vert^{2}}\,,
\end{equation}
and we obtain that condition $(\B_\zs{2})$ holds with $\gamma=1$ and
\begin{equation}
\label{def_h_theta_max_+++}
h(y,x)= \c_\zs{*}(2\theta_\zs{max}+\vert y\vert), \quad\quad \theta_\zs{max}=\max_\zs{u\in \Theta}\vert u\vert\,.
\end{equation}
Moreover, note that in this case
$$
\wt{g}(u,x)=\frac{1}{2}\,\vert G^{-1/2}(x)(u-A_\zs{0})x\vert=\frac{1}{2}\,x^{\top}(u-A_\zs{0})^{\top}G^{-1}(x)(u-A_\zs{0})x\,.
$$
The bound \eqref{sec:Ex.6-0n.1} implies that $g^{*}=\sup_\zs{x\in\bbr^{p}}\,\sup_\zs{\theta\in\K}\,\wt{g}(\theta,x)<\infty$.

Now, as in  Example 4 in \citep{PergTarSISP2016} choosing $\V(x)=\upsilon^{*}\,[1+(x^{\top}Tx)^{\delta}]$ with $\upsilon^{*}=1+g^{*}$ and any fixed $0<\delta\le 1$
yields condition ($\C_\zs{1}$).  Moreover, for any $r>0$ and  $\delta r\le 2$
\begin{equation}
 \label{upper_bound_moment_111++}
 \sup_\zs{x\in\bbr}\, \sup_\zs{\theta\in\K} \frac{\sup_\zs{j\ge 1}\E^{\theta}_\zs{x}\vert X_\zs{j}\vert^{\delta r}}{1+\vert x\vert^{\delta r}} <\infty,
 \qquad \sup_\zs{j\ge 1}\check{\E}\,\vert X_\zs{j}\vert^{\delta r} <\infty \,,
 \end{equation}
where $\check{\E}$ denotes the expectation with respect to the distribution $\check{\P}$ when $\nu=\infty$. Inequalities \eqref{upper_bound_moment_111++} imply ($\C_\zs{2}$) with $\q=\delta r$. 
Therefore,  taking into account that $\delta$ can be very close to zero and using Theorem~\ref{Th.sec:Mrk.1} we get that for any $r>0$ and any compact set 
$\K\subset \Theta\setminus \{A_\zs{0}\}$ condition $(\A^{*}_\zs{2}(r))$ holds with $I_\zs{\theta}=\E^{\theta}[\wt{g}(\theta,\varsigma_\zs{\theta})]$.

\end{example}

\begin{example}{(Change in the correlation coefficients of the AR($p$) model).}\label{ssec:ARp}
Consider the problem of detecting the change of the correlation coefficient in the $p$th order AR process,  assuming that for $n\ge 1$
\begin{equation}\label{sec:Ex.7}
X_\zs{n} =a_\zs{1,n}\,X_\zs{n-1}+\ldots+a_\zs{p,n}\,X_\zs{n-p}+w_\zs{n}\,,
\end{equation}
where $a_\zs{i,n}=a_\zs{i}\Ind{n \le \nu}+\theta_\zs{i}\Ind{n > \nu}$ and  $(w_\zs{n})_\zs{n\ge 1}$ are i.i.d.\ Gaussian random variables with $\EV [w_\zs{1}]=0$, $\EV[w^{2}_\zs{1}]=1$.  
In the sequel, we use the notation  $\a=(a_\zs{1},\ldots,a_\zs{p})^{\top}$ and $\thetab=(\theta_\zs{1},\ldots,\theta_\zs{p})^{\top}$.The process \eqref{sec:Ex.7} is not Markov, but the $p$-dimensional process  
\begin{equation}\label{sec:Ex.7-1n}
\Phi_\zs{n}=(X_\zs{n},\ldots,X_\zs{n-p+1})^{\top}\in\bbr^{p}
\end{equation}
is Markov.  Note that for $n > \nu$
\begin{equation*}
\Phi_\zs{n}=A\Phi_\zs{n-1}+\wt{w}_\zs{n}\,,
\end{equation*}
where 
\[
A=A(\thetab)= \begin{pmatrix}
\theta_\zs{1} & \theta_2 & \dots & \theta_\zs{p}\\
1 & 0 &  \dots & 0\\
\vdots & \vdots &  \ddots & \vdots \\
0 & 0 &\dots 1&0 
\end{pmatrix} ,
\quad\quad
\wt{w}_\zs{n}=(w_\zs{n},0,\ldots,0)\in\bbr^{p}\,.
\]
It is clear that
\[
\E[\wt{w}_\zs{n}\,\wt{w}^{^{\top}}_\zs{n}] =B=
\begin{pmatrix}
1 & \dots & 0\\
\vdots & \ddots & \vdots\\
0 & \dots & 0
\end{pmatrix} .
\]

Assume that the vectors $\a$ and $\thetab$ belong to the set 
\begin{equation*}
\Theta=\{u\in \bbr^{p}\,:\, \max_\zs{1\le j\le p}\,\vert\e_\zs{j}(A(u))\vert<1\}\,,
\end{equation*}
where $\e_\zs{j}(A)$ denotes the $j$th eigenvalue for the matrix $A$. Note that, in this case, for any $u$ from some compact set $\K\subset\Theta$
and any $y=(y_\zs{1},\ldots,y_\zs{p})^{\top}\in\bbr^{p}$ and $x=(x_\zs{1},\ldots,x_\zs{p})^{\top}\in\bbr^{p}$ the function 
\begin{equation*}
g(u,y,x)=y_\zs{1}(u-\a)^{\top}x+\frac{(\a^{\top} x)^{2}-(u^{\top}x)^{2}}{2}\,.
\end{equation*}

Obviously, it follows that condition $(\B_\zs{2})$ holds with $\gamma=1$ and
$$
h(y,x)= y^{2}_\zs{1}+(1+2\theta_\zs{max})\vert x\vert^{2}\,,
$$
where $\theta_\zs{max}$ is defined in \eqref{def_h_theta_max_+++}.

For any $\thetab\in\Theta$, the process $(\Phi_\zs{n})_\zs{n>\nu+p}$ is ergodic with the normal stationary distribution
\begin{equation}\label{erg_gaus.00}
\varsigma_\zs{\thetab}=\sum_\zs{n\ge 1}A^{n-1}\,\wt{w}_\zs{n}\sim \Nc(0,\F),
\quad\quad \F=\F(\thetab) =\sum_\zs{n\ge 0}A^{n}\,B\,(A^{\top})^{n}\,.
\end{equation}
Obviously, condition (C$1.1$) does not hold for the process \eqref{sec:Ex.7-1n}.  To fulfill this condition
we replace this process by the embedded homogeneous Markov process $\wt{\Phi}_\zs{\iota,n}=\Phi_\zs{np+\iota}$ for some $0\le \iota\le p-1$.  
This process can be represented as 
\begin{equation}\label{sec:A.7}
\wt{\Phi}_\zs{\iota,n}=A^{p}\wt{\Phi}_\zs{\iota,n-1}+\zeta_\zs{\iota,n}, \quad\quad \zeta_\zs{\iota,n}=\sum^{p-1}_\zs{j=0}\,A^{j}\,\wt{w}_\zs{np+\iota-j}\,.
\end{equation}
Clearly, $\zeta_\zs{\iota,n}$ is Gaussian with the parameters $(0,Q)$, where 
$$
Q=Q(\thetab)=\sum^{p-1}_\zs{j=0}\,A^{j}\, B\,(A^{\top})^{j}\,.
$$
One can check directly that this matrix is positive definite. Moreover, one can check directly that for any $\thetab\in\Theta$ and for any $0\le \iota\le p-1$ the process \eqref{sec:A.7} 
is ergodic with the same ergodic distribution given in \eqref{erg_gaus.00}. 

Now, for any fixed $0<\delta\le 1$ we define the $\bbr^{p}\to\bbr$  function 
\begin{equation}\label{sec:A.8}
\V^{\thetab}(x)= \check{\c}[1+(x^{\top}Tx)^{\delta}], \quad\quad T=T(\thetab)=\sum^{\infty}_\zs{l=0}\,(A^{\top})^{pl}\,A^{pl}\,,
\end{equation}
where  $\check{\c}\ge 1$ will be specified later. Let for any fixed compact set $\K\subset \Theta\setminus \{\a\}$
 $$
 t_\zs{\max}=\max_\zs{u\in\K}\,\vert T(u)\vert, \quad\quad Q_\zs{*}= \max_\zs{u\in\K}\,\vert Q(u)\vert \,.
 $$ 
Obviously, $t_\zs{\max}>1$. Note that,  by the Jensen inequality, for any $0\le \iota<p$
\begin{align*}
\E^{\theta}\,\left[\V^{\thetab}(\wt{\Phi}_\zs{\iota,1})\,|\,\wt{\Phi}_\zs{\iota,0}=x\right]
\le  \check{\c}+\check{\c} \left(x^{\top}(A^{p})^{\top}T\,A^{p}x +\tr TQ\right)^{\delta} \le
\check{\c}+\check{\c} \left(x^{\top}(A^{p})^{\top}T\,A^{p}x t_\zs{\max} Q_\zs{*} \right)^{\delta}\,.
\end{align*}
Also,
$$
x^{\top}Tx\ge \vert x\vert^{2}, \quad\quad \frac{x^{\top}(A^{p})^{\top}TA^{p}x}{x^{\top}Tx}=1-\frac{|x|^{2}}{x^{\top}Tx}\le 1-\frac{1}{t_\zs{max}}=t_\zs{*}<1\,.
$$
So,  taking into account that $(\vert a\vert+\vert b\vert)^{\delta}\le \vert a\vert^{\delta}+\vert b\vert^{\delta}$ for $0<\delta\le 1$,
we obtain 
$$
\E^{\theta}\,\left[\V^{\thetab}(\wt{\Phi}_\zs{\iota,1})\,|\,\wt{\Phi}_\zs{\iota,0}=x\right] \le
\check{\c}+\check{\c} \left[ t^{\delta}_\zs{*}(x^{\top}Tx)^{\delta} +(t_\zs{\max}Q_\zs{*})^{\delta} \right]\,.
$$
Putting
$$
N_\zs{*}=\left(\frac{2(1+t^{\delta}_\zs{\max}Q^{\delta}_\zs{*})}{1-t^{\delta}_\zs{*}}\right)^{1/2\delta}, \quad\quad \rho=(1-t^{\delta}_\zs{*})/2\,,
$$
yields that, for $\vert x\vert\ge N_\zs{*}$, 
$$
\E^{\theta}\,\left[\V^{\thetab}(\wt{\Phi}_\zs{\iota,1})\,|\,\wt{\Phi}_\zs{\iota,0}=x\right]\le (1-\rho)\,\V^{\thetab}(x)\,.
$$
Hence, the Markov process \eqref{sec:A.7} satisfies  the drift inequality \eqref{drift_in_+++} with 
$$
C=\{x\in\bbr^{p}\,:\,\vert x\vert \le N_\zs{*}\}, \quad\quad D=\check{\c} (1+t^{\delta}_\zs{max}\,N^{2\delta}_\zs{*}+t^{\delta}_\zs{max}\,Q^{\delta}_\zs{*})\,.
$$
Next we need the minorizing measure in condition $(\C^{\prime}_1)$ on the Borel $\sigma$-field in $\bbr^{p}$. 
To this end, we define $\check{\nu}(\Gamma) ={\rm mes}(\Gamma\cap C)/{\rm mes}(C)$ for any Borel set $\Gamma$ in $\bbr^{p}$,
where ${\rm mes}(\cdot)$ is the Lebesgue measure in $\bbr^{p}$. Moreover,  note that
$$
\wt{h}(\thetab,x)=1+(\thetab^{\top}x)^{2}+(1+2\theta_\zs{max})\vert x\vert^{2} \le 1+(1+2\theta_\zs{max})\vert x\vert^{2}
$$
and
$$
\wt{g}(\thetab,x)=\frac{1}{2}\,\left[(\thetab-\a_\zs{0})^{\top}x\right]^{2}\le \,\theta^{2}_\zs{max}|x|^{2}\,.
$$
Therefore, choosing in \eqref{sec:A.8} $\check{\c}=1+2\theta_\zs{max} +\theta^{2}_\zs{max}$,
we obtain condition $(\C^{\prime}1.2)$.  Condition ($\C_\zs{2}$) can be checked in the same way as in Example~1
for any $r>0$  for which $0<\q=\delta r\le 2$. Therefore, taking into account that $\delta$ can be very close to zero, Theorem~\ref{Th.sec:Mrk.1'} implies that 
for any $r>0$ and any compact set $\K\subset \Theta\setminus \{\a\}$ condition $(\A^{*}_\zs{2}(r))$ holds with $I_\zs{\thetab}=\E^{\thetab}[\wt{g}(\thetab,\varsigma_\zs{\thetab})]$.
\end{example}

\section{Monte Carlo simulations}\label{sec:MC}

In this section, we provide Monte Carlo (MC) simulations for the AR($1$) model, which is a particular case of Example~\ref{ssec:ARp}  for $p=1$, i.e., $a_{1,n}=\theta_0\Ind{n \le \nu}+\theta \Ind{n > \nu}$
and $a_{2,n}=\cdots=a_{p,n}\equiv 0$ in \eqref{sec:Ex.7}. Let the pre-change value 
$\theta_0=0$ and the post-change value $\theta\in\Theta=\{\theta_1,\dots,\theta_N\}$, $-1<\theta_1 < \theta_2 < \cdots  < \theta_N <1$, $\theta_i\neq 0$, and write
\[
L_n^\theta(X_n, X_{n-1}) = \exp\set{\theta X_n  X_{n-1} - \frac{\theta^2 X_{n-1}^2}{2}}, \quad n \ge 1.
\]
The WSR stopping time is written as
\[
T_a = \inf\set{n \ge 1: \log \brcs{ \sum_{j=1}^N W(\theta_j) R_n(\theta_j)} \ge a},
\]
where the SR statistic $R_n(\theta)$ tuned to $\theta$ satisfies the recursion
\[
R_{n+1}(\theta) = [1+R_n(\theta)] L_{n+1}^\theta(X_{n+1},X_{n}), \quad n \ge 0, \quad R_0(\theta)=0.
\]
Thus, the WSR procedure can be easily implemented.

The information number $I_\theta=\theta^2/[2(1-\theta^2)]$, so the first-order approximation \eqref{ASapproxBayes} yields the following approximate formula for the average delay to detection
$\ADD_{\nu,\theta_j}(T_a)=\EV_{\nu,\theta}(T_a-\nu | T_a>\nu)$:
\begin{equation}\label{ADDWSR}
\ADD_{\nu, \theta} (T_a)\approx  \ADD_{\nu, \theta}^{app} (T_a)=\frac{2 (1-\theta^2) a }{\theta^2} .
\end{equation}
In the MC simulations, we set 
\[
\Theta=\{-0.9, -0.8, -0.7, -0.6, -0.5, -0,4, -0.3, -0.2, -0.1, 0.1, 0.2, 0.3, 0.4, 0.5, 0.6, 0.7, 0.8, 0.9\}
\] 
and uniform prior $W(\theta_j)=1/18$, $j=1,\dots,18$. 

\begin{table}[h!]
\begin{center}
\caption{Operating characteristics of the WSR and SR detection procedures. Results of MS simulations with $10^6$ runs for the probability of false alarm $\beta=0.01$ and the change points $\nu=0$ and 
$\nu=10$. The worst change point is $\nu=0$. \label{t:t1}}
\begin{tabular}{ |c||c|c|c||c|c|c||c| } 
 \hline
 \multicolumn{8}{|c|}{$\beta = 0.01, \nu = 0$} \\
 \hline
 $\theta$ & $e^a$ & $\ADD_{\nu,\theta}(T_a)$ & $\LCPFA(T_a)$ & $B$ & $\ADD_{\nu,\theta}(T^*_B)$ & $\LCPFA(T_B^*)$ & $ \ADD_{\nu, \theta}^{app} (T_a)$ \\
 \hline
 0.9 & 395 & 11.74 & 0.0080 & 791 & 11.08 & 0.0079 & 2.81 \\
 \hline
 0.8 & 420 & 14.72 & 0.0073 & 791 & 13.72 & 0.0073 & 6.80 \\
 \hline
 0.7 & 440 & 18.97 & 0.0070 & 791 & 17.52 & 0.0071 & 12.67 \\
 \hline
 0.6 & 470 & 25.32 & 0.0065 & 791 & 23.15 & 0.0065 & 21.88 \\
 \hline
 0.5 & 595 & 36.35 & 0.0049 & 791 & 31.84 & 0.0049 & 38.33 \\
 \hline
 0.4 & 1040 & 59.57 & 0.0024 & 791 & 45.88 & 0.0025 & 72.94 \\
 \hline
 \hline
 \multicolumn{8}{|c|}{$\beta = 0.01, \nu = 10$} \\
 \hline
 $\theta$ & $e^a$ & $\ADD_{\nu,\theta}(T_a)$ & $\LCPFA(T_a)$ & $B$ & $\ADD_{\nu,\theta}(T^*_B)$ & $\LCPFA(T_B^*)$ & $ \ADD_{\nu, \theta}^{app} (T_a)$  \\
 \hline
 0.9 & 395 & 10.05 & 0.0080 & 791 & 9.62 & 0.0079 & 2.81 \\
 \hline
 0.8 & 420 & 12.72 & 0.0073 & 791 & 11.98 & 0.0073 & 6.80 \\
 \hline
 0.7 & 440 & 16.59 & 0.0070 & 791 & 15.30 & 0.0071 & 12.67 \\
 \hline
 0.6 & 470 & 22.55 & 0.0065 & 791 & 20.34 & 0.0065 & 21.88 \\
 \hline
 0.5 & 595 & 32.96 & 0.0049 & 791 & 28.01 & 0.0049 & 38.33 \\
 \hline
 0.4 & 1040 & 55.34 & 0.0024 & 791 & 40.83 & 0.0025 & 72.94 \\
 \hline
\end{tabular}
\end{center}
\end{table}

The results are presented in Table~\ref{t:t1} for the upper bound on the maximal local conditional probability of false alarm ($\LCPFA$) $\beta=0.01$ 
and the number of MC runs  $10^6$. In the table, we compare operating characteristics of the WSR rule $T_a$ with that of the  SR rule 
\[
T_B^*(\theta) = \inf\set{n\ge 1: R_n(\theta) \ge B} 
\]
tuned to the true value of the post-change parameter $\theta$,  i.e., assuming that the post-change parameter $\theta$ is known and equals to one of the values shown in the table.  
Thresholds $a$ and $B$ (shown in the table) were selected in such a way that the maximal 
probabilities of false alarm of both rules ($\LCPFA(T_a)$ and $\LCPFA(T^*_B)$) were practically the same.  It is seen that for relatively large values of the post-change parameter, $\theta \ge 0.6$,
the SR rule only slightly outperforms the WSR rule, but for small parameter values (i.e., for close hypotheses) the difference becomes quite substantial.  The worst change point is $\nu=0$, as expected.
Also, the first-order approximation \eqref{ADDWSR} is not too accurate, especially for small and large parameter values.

\section{Remarks}\label{sec:Rem}

1. Despite the fact that the WSR procedure is first-order asymptotically optimal for practically arbitrary distribution $W(\theta)$ that satisfies condition ($\C_W$), for practical purposes its choice 
may be important. In fact, selection of the weight $W$ affects the higher-order asymptotic performance, and therefore, the real performance of the detection procedure.
For example, if the set $\Theta$ is continuous, one has to avoid $W(\theta)$ that concentrates in the vicinity of a specific parameter value $\theta_1$ since in this case the WSR procedure will be nearly 
optimal at and in the vicinity of $\theta_1$ but will not have a good performance for other parameter values. The choice of $W(\theta)$ is also related to the computational issue. 
It is reasonable to select the weight as to be in the class of conjugate priors, if possible, or to select a uniform prior if $\Theta$ is compact. A substantial simplification occurs when 
$\Theta=\{\theta_1,\dots, \theta_N\}$ is a finite discrete set.  If the observations are i.i.d., then in the discrete case, it is possible to find an optimal (in a certain sense) weight using the approach proposed by 
Fellouris and Tartakovsky~\cite{FellourisTartakovsky-SS2013} for the hypothesis testing problem.

2.  The traditional constraint on the false alarm risk in minimax change-point detection problems is the lower bound on the average run length to false alarm (ARL2FA) $\EV_\infty [\tau] \ge \gamma \ge 1$. This
measure of false alarms makes sense when the distribution of the stopping time $\tau$ (in our case of $T_a$) is approximately geometric. This is typically the case (at least asymptotically as $a\to\infty$) 
for i.i.d.\ data models \cite{PollakTartakovskyTPA09, Yakir-AS95}. However, apart from the i.i.d.\ case, there is no result on the asymptotic distribution of the stopping time 
$T_a$ (as $a\to\infty$), so for general non-i.i.d.\ models of interest in the present paper this is not necessarily true. Therefore, the usefulness of the ARL2FA is under the question, as discussed in detail in 
\cite{Mei-SQA08,TartakovskyIEEECDC05,Tartakovsky-SQA08a,TNB_book2014}. In fact, in general, large values of the ARL2FA do not guarantee small values of the maximal local PFA 
$\sup_{k \ge 1} \Pb_\infty(\tau < k+m | \tau \ge k)$. But the opposite is always true since the maximal local PFA is a more stringent false alarm measure in 
the sense that if it is small, then the ARL2FA is necessarily large. This argument motivated us to consider the maximal local PFA instead of conventional ARL2FA.

\section*{Acknowledgements}

The work of the first author was  partially supported by  the RSF grant 17-11-01049 (National Research Tomsk State University), the Russian Federal Professorship program (project no. 1.472.2016/1.4)
and by the research project no. 2.3208.2017/4.6 (the  Ministry of Education and Science of the Russian Federation).  The work of the second author  was supported in part by Russian 
Science Foundation grant No.\ 18-19-00452 and by the Russian Federation Ministry of Science and Education Arctic program at the Moscow Institute of Physics and Technology. 

We would like to thank referees for useful comments that have improved the article as well as Editor-in-Chief Dr. Dietrich von Rosen for excellent handling of the manuscript. 
We also thank the student Valentin Spivak of the Moscow Institute of Physics and Technology for helping with MC simulations.





\renewcommand{\thetheorem}{A.\arabic{theorem}}
\setcounter{theorem}{0}

\renewcommand{\theproposition}{A.\arabic{proposition}}
\setcounter{proposition}{0}

\section*{Appendix. Auxiliary non-asymptotic bounds for the concentration inequalities}\label{B}

\subsection*{Correlation inequality}

The following proposition provides the important correlation inequality \citep{GaPe13}.

\begin{proposition}\label{Pr.sec:Bi.1}
 Let $(\Omega,\Fc,(\Fc_\zs{j})_\zs{1\le j\le n},\Pb)$ be a filtered probability space and
$(u_\zs{j}, \Fc_\zs{j})_\zs{1\le j\le n}$ be a sequence of random variables such that $\max_{1\le j\le n}\,\EV\,|u_{j}|^{q}<\infty$
for some $q\ge 2$. Define
\begin{equation*}
\check{b}_\zs{j,n}(q)= \left\{ \EV\,\left[ |u_\zs{j}|\,\sum^{n}_\zs{k=j} |\EV\,(u_\zs{k}|\Fc_\zs{j})|\right]^{q/2}\right\}^{2/q}\,.
\end{equation*}
Then
$$
\EV\, \Big | \sum^{n}_\zs{j=1}\,u_\zs{j} \Big |^{q} \le\,(2q)^{q/2} \left(\sum^{n}_\zs{j=1}\,\check{b}_\zs{j,n}(q)\right)^{q/2}\,.
$$
\end{proposition}

\subsection*{Uniform geometric ergodicity for homogeneous Markov processes}\label{subsec:AMc}

We recall some definitions from  \citep{GaPe14}
for a  homogeneous Markov process $(X_\zs{n})_\zs{n\ge 0}$ defined on a measurable state space
$(\Xc, \Bc(\Xc))$. Denote by $(P^{\theta}(\cdot,\cdot))_\zs{\theta\in\Theta}$ the transition probability family 
of this process, i.e., for any $A\in\Bc(\Xc), x\in\Xc$,
\begin{equation*}\label{subsec:AMc.1-1}
P^{\theta}(x,A)\,=\,\Pb^{\theta}_\zs{x}(X_\zs{1}\in A)=\Pb^{\theta}(X_\zs{1}\in A|X_\zs{0}\,=\,x)\,.
\end{equation*}
The $n-$step transition probability is $P^{n,\theta}(x,A)\,=\,\Pb^{\theta}_\zs{x}(X_\zs{n}\in A)$.

We recall that a measure $\lambda$ on $\Bc(\Xc))$ is called 
{\it invariant} (or {\em stationary} or {\em ergodic}) for this process if,
for any $A\in\Bc(\Xc)$, 
\begin{equation*}
\lambda^{\theta}(A)\,=\,\int_{\Xc}\,P^{\theta}(x,A)\lambda(\d x)\,.
\end{equation*}

\noindent If there exists an invariant positive measure 
$\lambda^{\theta}$ with $\lambda^{\theta}(\Xc)\,=\,1$ then the process is
called {\it positive}.

Assume that the process $(X_\zs{n})_\zs{n\ge 0}$ satisfies the following {\em minorization} condition: \vspace{2mm}

\noindent $(\D_\zs{1})$
{\em
There exist  $\delta>0$, a set $C\in\Bc(\Xc)$ and a probability measure $\varsigma$ on $\Bc(\Xc)$ with
$\varsigma(C)=1$, such that
$$
\inf_\zs{\theta\in\Theta}\,\left(\inf_\zs{x\in C}P^{\theta}(x,A)-\delta\,\varsigma(A)\right)>0
$$
 for any $A\in \Bc(\Xc)$, for which $\varsigma(A)>0$.
}

\noindent 
Obviously, this condition implies that $\eta\,=\,\inf_\zs{\theta\in\Theta}\,\inf_\zs{x \in C}\,P^{\theta}(x,C)-\delta>0$. Now we impose the {\em uniform drift} condition. \vspace{2mm}

\noindent $(\D_\zs{2})$
 {\em
There exist some constants  $0<\rho<1$ and $D\ge 1$ such that for any $\theta\in\Theta$ there exist a $\Xc\to [1,\infty)$ function $\V^\theta$  and a set $C$ from $\Bc(\Xc)$ such that 
$$
 \V^{*}=\sup_\zs{\theta\in\Theta}\sup_\zs{x\in C}|\V^\theta(x)|<\infty
$$
and}
\begin{equation*}
\sup_\zs{\theta\in\Theta}\sup_\zs{x\in \Xc}\left\{\E^{\theta}_\zs{x} \left[\V^\theta(X_\zs{1})\right]\,-\,(1-\rho)\V^\theta(x)\,+\,D\Ind{C}(x)\right\}\le 0\,.
\end{equation*}
In this case, we call $\V^\theta$ the Lyapunov function. We use the following theorem from \citep{GaPe14}.

\begin{theorem} \label{Th.AMc.1}
Let $(X_\zs{n})_\zs{n\ge 0}$ be a homogeneous  Markov process satisfying conditions $(\D_\zs{1})$ and $(\D_\zs{2})$ with the same set 
$C\in\Bc(\Xc)$. Then $(X_\zs{n})_\zs{n\ge 0}$ is a positive uniform geometric ergodic process, i.e.,
\begin{equation*}
\sup_\zs{n\ge 0}\, e^{\kappa^{*} n}\, \sup_\zs{x\in\Xc}\, \sup_\zs{\theta\in\Theta} \sup_\zs{0\le g\le \V^\theta}\frac{1}{\V^\theta(x)}\left |\E^{\theta}_\zs{x}[g(X_\zs{n})] -\lambda^\theta(\wt{g})\right | \le R^{*}
\end{equation*}
for some positive constants $\kappa^{*}$ and  $R^{*}$ which are given in {\rm \citep{GaPe14}}.
\end{theorem}



\section*{References}



\end{document}